\documentclass[11pt]{article}
\usepackage{booktabs} 
\usepackage{fullpage, times}
\usepackage{times}

\usepackage{amsmath}
\usepackage{color}
\usepackage{amsfonts}
\usepackage{amsthm}

\usepackage{amssymb}
\usepackage{bbm}
\usepackage{graphicx}

\usepackage{algorithm}
\usepackage{algorithmic}

\usepackage{mathtools}

\usepackage{enumitem}

\usepackage{verbatim}

\newtheorem{definition}{Definition}
\newtheorem{lemma}{Lemma}
\newtheorem{corollary}{Corollary}
\newtheorem{theorem}{Theorem}

\newtheorem*{conjecture*}{Conjecture}

{
\vspace{0.01cm}
\begin{center}
	\begin{boxedminipage}{0.8\linewidth}
		\begin{center}
			\textbf{\texttt{#1}}
		\end{center}
	} %
	{
	\end{boxedminipage}
\end{center}
 \vspace{0.3cm}
}

\newcommand{\bx}{{\mathbf x}}

\newcommand{\bw}{{\mathbf w}}

\newcommand{\bq}{{\mathbf  q}}

\newcommand{\bu}{{\mathbf u}}

\newcommand{\be}{{\mathbf e}}

\newcommand{\bs}{{\mathbf{s}}}
\newcommand{\bt}{{\mathbf{t}}}
\newcommand{\bi}{{\mathbf{i}}}

\newcommand{\balpha}{\boldsymbol{\alpha}}
\newcommand{\bpsi}{\boldsymbol{\psi}}

\newcommand{\bbeta}{\boldsymbol{\eta}}

\newcommand{\bE}{\mathbb{E}}

\newcommand{\reals}{\mathbb{R}}

\newcommand{\bN}{\mathbb{N}}

\newcommand{\sgn}{{\mathrm{sgn}}}

\newcommand{\spn}[1]{\mathrm{span}\left\{ #1\right\}}

\newcommand{\inner}[1]{\langle #1 \rangle}

\newcommand{\eqdef}{\stackrel{\vartriangle}{=}}

\newcommand{\cA}{\mathcal{A}}

\newcommand{\cD}{\mathcal{D}}
\newcommand{\cE}{\mathcal{E}}
\newcommand{\cF}{\mathcal{F}}

\newcommand{\cH}{\mathcal{H}}
\newcommand{\cI}{\mathcal{I}}
\newcommand{\cJ}{\mathcal{J}}

\newcommand{\cM}{\mathcal{M}}
\newcommand{\cN}{\mathcal{N}}
\newcommand{\cO}{\mathcal{O}}
\newcommand{\cP}{\mathcal{P}}
\newcommand{\cQ}{\mathcal{Q}}

\newcommand{\cS}{\mathcal{S}}

\newcommand{\cU}{\mathcal{U}}

\newcommand{\cX}{\mathcal{X}}

\newcommand{\bones}{\mathbbm{1}}
\renewenvironment{proof}{\par\noindent{\bf Proof\ }}{\hfill\BlackBox\\[2mm]}
\newcommand{\BlackBox}{\rule{1.5ex}{1.5ex}}

\newcommand{\RNum}[1]{\uppercase\expandafter{\romannumeral #1\relax}}
\makeatletter
\def\moverlay{\mathpalette\mov@rlay}
\def\mov@rlay#1#2{\leavevmode\vtop{%
   \baselineskip\z@skip \lineskiplimit-\maxdimen
   \ialign{\hfil$\m@th#1##$\hfil\cr#2\crcr}}}
\newcommand{\charfusion}[3][\mathord]{
    #1{\ifx#1\mathop\vphantom{#2}\fi
        \mathpalette\mov@rlay{#2\cr#3}
      }
    \ifx#1\mathop\expandafter\displaylimits\fi}
\makeatother

\DeclareMathOperator*{\argmin}{argmin} 

\renewcommand{\eqref}[1]{Equation~(\ref{#1})}
\newcommand{\ineqref}[1]{Inequality~(\ref{#1})}

\newcommand{\figref}[1]{Figure~\ref{#1}}

\newcommand{\thmref}[1]{Theorem~\ref{#1}}
\newcommand{\lemref}[1]{Lemma~\ref{#1}}
\newcommand{\defref}[1]{Definition~\ref{#1}}
\newcommand{\corref}[1]{Corollary~\ref{#1}}

\newcommand{\coursename}{(67577) INTRODUCTION TO MACHINE LEARNING}

\newcommand{\handout}[5]{
   \renewcommand{\thepage}{#1-\arabic{page}}
   \noindent
   \begin{center}
   \framebox{
      \vbox{
    \hbox to 5.78in { {\bf \coursename}
         \hfill #2 }
       \vspace{4mm}
       \hbox to 5.78in { {\Large \hfill #5  \hfill} }
       \vspace{2mm}
       \hbox to 5.78in { {\it #3 \hfill #4} }
      }
   }
   \end{center}
   \vspace*{4mm}
}

\newcommand{\norm}[1]{\left\Vert#1\right\Vert}

\newcommand{\circpar}[1]{\left( #1 \right)}

\newcommand{\dom}{\text{dom}}

\newcommand{\mymat}[1]{\circpar{\begin{array}{cccccccccccccccccc} #1 \end{array}}}

\newcommand{\bigO}[1]{\mathcal{O}{\left(#1\right)}}

\usepackage[utf8]{inputenc} 
\usepackage[T1]{fontenc}    
\usepackage{hyperref}       
\usepackage{url}            
\usepackage{booktabs}       
\usepackage{amsfonts}       
\usepackage{nicefrac}       
\usepackage{microtype}      
\usepackage{epstopdf}

\title{Dimension-Free Iteration Complexity of Finite Sum Optimization Problems}

\author{
Yossi Arjevani \\
Weizmann Institute of Science\\
Rehovot 7610001, Israel \\
\texttt{yossi.arjevani@weizmann.ac.il} \\
\and
Ohad Shamir\\
Weizmann Institute of Science\\
Rehovot 7610001, Israel \\
\texttt{ohad.shamir@weizmann.ac.il} \\
}	
\date{}
\begin{document}

\maketitle

\begin{abstract}
Many canonical machine learning problems boil down to a convex optimization problem with a finite sum structure. However, whereas much progress has been made in developing faster algorithms for this setting, the inherent limitations of these problems are not satisfactorily addressed by existing lower bounds. Indeed, current bounds focus on first-order optimization algorithms, and only apply in the often unrealistic regime where the number of iterations is less than $\cO(d/n)$ (where $d$ is the dimension and $n$ is the number of samples). In this work, we extend the framework of Arjevani et al. \cite{arjevani2015lower,arjevani2016iteration} to provide new lower bounds, which are dimension-free, and go beyond the assumptions of current bounds, thereby covering standard finite sum optimization methods, e.g., SAG, SAGA, SVRG, SDCA without duality, as well as stochastic coordinate-descent methods, such as SDCA and accelerated proximal SDCA.

\end{abstract}
\section{Introduction}

Many machine learning tasks reduce to \emph{Finite Sum Minimization (FSM)} problems of the form 
\begin{align} \label{opt:finite_sum}
	\min_{\bx\in\reals^d} F(\bw) \coloneqq \frac{1}{n}\sum_{i=1}^n f_i(\bw),
\end{align} 
where $f_i$ are $L$-smooth and $\mu$-strongly convex. In recent years, a major breakthrough was made when a linear convergence rate was established for this setting (SAG \cite{roux2012stochastic} and SDCA \cite{shalev2013stochastic}), and since then, many methods have been developed to achieve better convergence rate. However, whereas a large body of literature is devoted for upper bounds, the optimal convergence rate with respect to the problem parameters is not quite settled.

Let us discuss existing lower bounds for this setting, along with their shortcomings, in detail. One approach to obtain lower bounds for this setting is to consider the average of carefully handcrafted functions defined on $n$ disjoint sets of variables. This approach was taken by Agarwal and Bottou \cite{agarwal2014lower} 
who derived a lower bound for FSM under the first-order oracle model (see Nemirovsky and Yudin \cite{nemirovskyproblem}). In this model, optimization algorithms are assumed to access a given function by issuing queries to an external first-order oracle procedure. Upon receiving a query point in the problem domain, the oracle reports the corresponding function value and gradient. The construction used by Agarwal and Bottou consisted of $n$ different quadratic functions which are adversarially determined based on the first-order queries being issued during the optimization process. The resulting bound in this case does not apply to stochastic algorithms, rendering it invalid for current state-of-the-art methods. Another instantiation of this approach was made by Lan \cite{lan2015optimal} who considered $n$ disjoint copies of a quadratic function proposed by Nesterov in \cite[Section 2.1.2]{nesterov2004introductory}. This technique is based on the assumption that any iterate generated by the optimization algorithm lies in the span of previously acquired gradients. This assumption is rather permissive and is satisfied by many first-order algorithms, e.g., SAG and SAGA \cite{defazio2014saga}. However, the lower bound stated in the paper faces limitations in a few aspects. First, the validity of the derived bound is restricted to $d/n$ iterations. In many datasets, even if $d,n$ are very large, $d/n$ is quite small. Accordingly, the admissible regime of the lower bound is often not very interesting. Secondly, it is not clear how the proposed construction can be expressed as a Regularized Loss Minimization (RLM) problem with linear predictors (see Section \ref{section:lb_dual_rlm}). This suggests that methods specialized in dual RLM problems, such as SDCA and accelerated proximal SDCA \cite{shalev2016accelerated}, can not be addressed by this bound. Thirdly, at least the formal theorem requires assumptions (such as querying in the span of previous gradients, or sampling from a fixed distribution over the individual functions), which are not met by some state-of-the-art methods, such as coordinate descent methods, SVRG \cite{johnson2013accelerating} and without-replacements sampling algorithms \cite{recht2012beneath}.

Another relevant approach in this setting is to model the functional form of the update rules. This approach was taken by Arjevani et al. \cite{arjevani2015lower} where new iterates are assumed to be generated by a recurrent application of some fixed linear transformation. Although this method applies to SDCA and produces a tight lower bound of $\tilde\Omega((n+1/\lambda)\ln(1/\epsilon))$, its scope is rather limited. In recent work, Arjevani and Shamir \cite{arjevani2016iteration} considerably generalized parts of this framework by introducing the class of first-order \emph{oblivious} optimization algorithms, whose step sizes are scheduled regardless of the function under consideration, and deriving tight lower bounds for general smooth convex minimization problems (note that obliviousness rules out, e.g., quasi-Newton methods where gradients obtained at each iteration are multiplied by matrices which strictly depend on the function at hand, see \defref{definition:pcli} below).

In this work, building upon the framework of oblivious algorithms, we take a somewhat more abstract point of view which allows us to easily incorporate coordinate-descent methods, as well as stochastic algorithms. Our framework subsumes the vast majority of optimization methods for machine learning problems, in particular, it applies to SDCA, accelerated proximal SDCA, SDCA without duality \cite{shalev2015sdca}, SAG, SAGA, SVRG and acceleration schemes \cite{frostig2015regularizing,lin2015universal}), as well as	for a large number of methods for smooth convex optimization (i.e., FSM with $n=1$), e.g., (stochastic) Gradient descent (GD), Accelerated Gradient Descent (AGD, \cite{nesterov2004introductory}), the Heavy-Ball method (HB, \cite{polyak1987introduction}) and stochastic coordinate descent.

Under this structural assumption, we derive lower bounds for FSM (\ref{opt:finite_sum}), according to which the iteration complexity, i.e., the number of iterations required to obtain an $\epsilon$-optimal solution in terms of function value,  
is at least\footnote{Following standard conventions, here tilde notation hides logarithmic factors in the parameters of a given class of optimization problems, e.g., smoothness parameter and number of components.} 
\begin{align} \label{FSM_RLM_lower_bound}
	\tilde\Omega(n+\sqrt{n(\kappa-1)}\ln(1/\epsilon)),
\end{align}
where $\kappa$  denotes the condition number of $F(\bw)$ (that is, the smoothness parameter over the strong convexity parameter). To the best of our knowledge, this is the first tight lower bound to address \emph{all} the algorithms mentioned above. Moreover, our bound is dimension-free and thus apply to settings in machine learning which are not covered in the current literature (e.g., when $n$ is $\Omega(d)$). We also derive a dimension-free nearly-optimal lower bound for smooth convex optimization of 
$$\Omega\circpar{\circpar{L(\delta-2)/\epsilon}^{1/\delta}},~\delta\in(2,4),$$ 
which holds for any oblivious stochastic first-order algorithm. 
 It should be noted that 
our lower bounds remain valid under any source of randomness which may be introduced into the optimization process (by the oracle or by the optimization algorithm). In particular, our bounds hold in cases where the variance of the iterates produced by the algorithm converges to zero, a highly desirable property of optimization algorithms in this setting.

Two implications can be readily derived from this lower bound. First, 
obliviousness forms a real barrier for optimization algorithms, and whereas non-oblivious algorithms may achieve a super-linear convergence rate at latter stages of the optimization process (e.g., quasi-newton), or practically zero error after $\Theta(d)$ iterations (e.g. Center of Gravity method, MCG), oblivious algorithms are bound to linear convergence indefinitely, as demonstrated by \figref{figure:AGD_LBFGS}. We believe that this indicates that a major progress can be made in solving machine learning problems by employing non-oblivious methods for settings where $d \ll n$. It should be further noted that another major advantage of non-oblivious algorithms is their ability to obtain optimal convergence rates without an explicit specification of the problem parameters (e.g., \cite[Section 4.1]{arjevani2016iteration}).
\begin{figure}[H] \label{figure:AGD_LBFGS}
  \centering
     \includegraphics[scale=0.4,trim= 0 195 0 210,clip]{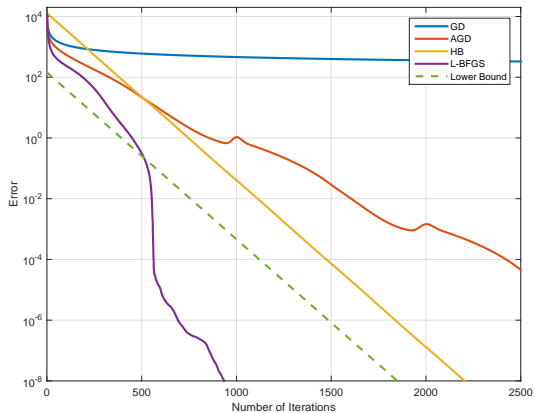}
  \caption{Comparison of first-order methods based on the function used by Nesterov in \cite[Section 2.1.2]{nesterov2004introductory} over $\reals^{500}$. Whereas L-BFGS (with a memory size of 100) achieves a super-linear convergence rate after $\Theta(d)$ iterations, the convergence rate of GD, AGD and HB remains linear as predicted by our bound.}
\end{figure}

Secondly, many practitioners have noticed that oftentimes sampling the individual functions without replacement at each iteration performs better than sampling with replacement (e.g., \cite{shalev2013stochastic,recht2012beneath}, see also  \cite{gurbuzbalaban2015random,
shamir2016without}). The fact that our lower bound holds regardless of how the individual functions are sampled and is attained using with-replacement sampling (e.g., accelerated proximal SDCA), implies that, in terms of iteration complexity, one should expect to gain no more than log factors in the problem parameters when using one method over the other (it is noteworthy that	 when comparing with and without replacement~samplings, apart from iteration complexity, other computational resources, such as limited communication in distributed settings \cite{arjevani2015communication}, may significantly affect the overall runtime). 

																				\section{Framework}

														\subsection{Motivation} \label{section:background}

Due to difficulties which arise when studying the complexity of general optimization problems under discrete computational models, it is common to analyze the computational hardness of optimization algorithms by modeling the way a given algorithm interacts with the problem instances (without limiting its computational resources). In the seminal work of Nemirovsky and Yudin \cite{nemirovskyproblem}, it is shown that algorithms which access the function at hand exclusively by querying a first-order oracle require at least 
\begin{align} \label{ineq:sqrtlb_lb}
	&\tilde{\Omega}\circpar{\min\left\{d, \sqrt{\kappa}\right\} \ln(1/\epsilon)},&\mu>0\\
	&\tilde{\Omega}{(\min\{d\ln(1/\epsilon),\sqrt{L/\epsilon}\})},&\mu=0\nonumber
\end{align}
oracle calls to obtain an $\epsilon$-optimal solution (note that, here and throughout this section we refer to FSM problems with $n=1$). This lower bound is tight and its dimension-free part is attained by Nesterov's well-known accelerated gradient descent, and by MCG otherwise. The fact that this approach is based on information considerations alone is very appealing and renders it valid for any first-order algorithm. However, discarding the resources needed for executing a given algorithm, in particular the per-iteration cost (in time and space), the complexity boundaries drawn by this approach are too crude from a computational point of view. Indeed, the per-iteration cost of MCG, the only method known with oracle complexity of $\bigO{d\ln(1/\epsilon)}$, is excessively high, rendering it prohibitive for high-dimensional~problems. 

We are thus led into the question of how well can a given optimization algorithm perform assuming that its per-iteration cost is constrained? Arjevani et al. \cite{arjevani2015lower,arjevani2016iteration} adopted a more structural approach where instead of modeling how information regarding the function at hand is being collected, one models the update rules according to which iterates are being generated. Concretely, they proposed the framework of $p$-CLI optimization algorithms where, roughly speaking, new iterates are assumed to form linear combinations of the previous $p$ iterates and gradients, and the coefficients of these linear combinations are assumed to be either stationary (i.e., remain fixed throughout the optimization process) or oblivious. Based on this structural assumption, they showed that the iteration complexity of minimizing smooth and strongly convex functions is $\tilde{\Omega}(\sqrt{\kappa}\ln(1/\epsilon))$. The fact that this lower bound is stronger than (\ref{ineq:sqrtlb_lb}), 
in the sense that it does not depend on the dimension, confirms that controlling the functional form of the update rules allows one to derive tighter lower bounds. The framework of $p$-CLIs forms the nucleus of our formulation below.

															\subsection{Definitions} \label{section:framework}



When considering lower bounds one must be very precise as to the scope of optimization algorithms to which they apply. Below, we give formal definitions for oblivious stochastic CLI optimization algorithms and iteration complexity (which serves as a crude proxy for their computational~complexity).

\begin{definition}[Class of Optimization Problems] \label{definition:side_information}
A class of optimization problems is an ordered triple $(\cF,\cI,\cO)$, where $\cF$ is a family of functions defined over some linear space designated by $\dom\cF$, $\cI$ is the side-information given prior to the optimization process and $\cO_f:\dom\cF\times\Theta\to\dom\cF$ is a suitable oracle parametrized by some parameters set $\Theta$, i.e., an external procedure which upon receiving $\bx\in\dom\cF$ and $\theta\in\Theta$, returns some $\cO_f(\theta)\in\dom(\cF)$.
\end{definition}
For example, in FSM, $\cF$ contains functions as defined in (\ref{opt:finite_sum}), the side-information contains the smooth parameter $L$, the strong convexity parameter $\mu$ and the number of components $n$ (although it carries a crucial effect on the iteration complexity, e.g., \cite{arjevani2016iteration}, in this work, we shall ignore the side-information and assume that all the parameters of the class are given). We shall assume that both first-order and coordinate-descent oracles (see \ref{oracle:first_order},\ref{oracle:fsm_cd} below) are allowed to be used during the optimization process. Formally, this is done by introducing an additional parameter which indicates which oracle is being addressed. This added degree of freedom does not violate our lower bounds.

We now turn to rigorously define CLI optimization algorithms. Note that, compared with the definition of first-order $p$-CLIs provided in \cite{arjevani2016iteration}, here, in order to handle coordinate-descent and first-order oracles in a unified manner, we base our formulation on general oracle procedures.  

\begin{definition}[CLI]\label{definition:pcli} 
An optimization algorithm is called a Canonical Linear Iterative (CLI) optimization algorithm over a class of optimization problems $(\cF,\cI,\cO)$, if given an instance $f\in\cF$ and initialization points $\{\bw^{(0)}_{i}\}_{i\in\cJ}\subseteq\dom(\cF)$, where $\cJ$ is some index set, it operates by iteratively generating points such that for any $i\in\cJ$,
\begin{align} \label{assumption:dynamics}
	\bw^{(k+1)}_i \in \sum_{j\in\cJ}  \cO_f\circpar{\bw^{(k)}_j;\theta_{ij}^{(k)}}, \quad k=0,1,\dots 
\end{align} 
holds, where $\theta^{(k)}_{ij}\in\Theta$ are parameters chosen, stochastically or deterministically, by the algorithm, possibly depending on the side-information. If the parameters do not depend on previously acquired oracle answers, we say that the given algorithm is \emph{oblivious}. Lastly, algorithms with $|\cJ|\le p$, for some $p\in\bN$, are denoted by $p$-CLI.
\end{definition}	
Note that assigning different weights to different terms in (\ref{assumption:dynamics}) can be done through  $\theta^{(k)}_{ij}\in\Theta$ (e.g., oracle \ref{oracle:first_order} below). This allows a succinct definition for obliviosity. Lastly, we define \emph{iteration~complexity}.
\begin{definition}[Iteration Complexity]\label{definition:iteration_complexity}
The iteration complexity of a given CLI w.r.t. a given problem class $(\cF,\cI,\cO)$ is defined to be the minimal number of iterations $K$ such that 
\begin{align*}
	   \bE[f(\bw^{(k)}_1) - \min_{\bw\in\dom{\cF}} f(\bw)] <\epsilon,\quad \forall f\in\cF, k\ge K
\end{align*}
where the expectation is taken over all the randomness introduced into the optimization process (choosing $\bw_1^{(k)}$ merely serves as a convention and is not necessary for our bounds to hold). 
\end{definition}
%
%

						\subsection{Proof Technique - Deriving Lower Bounds via Approximation Theory} \label{section:approximation}

Consider the following parametrized class of $L$-smooth and $\mu$-strongly convex optimization problems,
\begin{align}\label{opt:approximation_toy_problem}
	\min_{x\in\reals} f_\eta(x)\coloneqq \frac{\eta w^2}{2}  - w,\quad \eta\in [\mu,L].
\end{align} 
Clearly, the minimizer of $f_\eta$ are $w^*(\eta)\coloneqq 1/\eta$, with norm bounded by $1/\mu$. For simplicity, we will consider a special case, namely, vanilla gradient descent (GD) with step size $1/L$, which produces new iterates as follows
\begin{align*}
	w^{(k+1)}(\eta)&=w^{(k)}(\eta) - \frac{1}{L} f'_\eta(w^{(k)}(\eta) )
	= \circpar{1- \frac\eta L}w^{(k)}(\eta)+\frac{1}{L}.
\end{align*}
Setting the initialization point to be $w^{(0)}(\eta)=0$, we derive an explicit expression for $w^{(k)}(\eta)$:
\begin{align}\label{eq:gd_iterations}
	w^{(k)}(\eta)
	&= \frac{1}{L}\sum_{i=0}^{k-1} (-1)^{i} \binom{k}{i+1} (\eta/L)^i.
\end{align}
\begin{figure}[H]
  \centering
     \includegraphics[scale=0.4,trim= 0 210 0 210,clip]{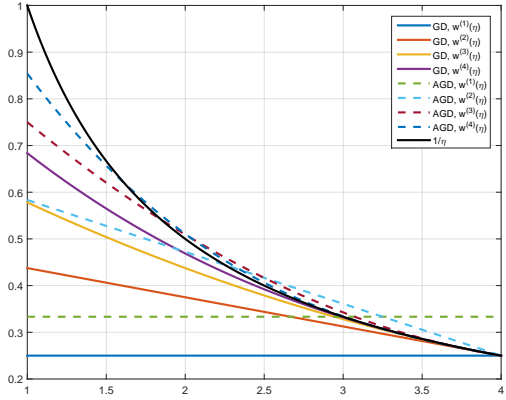} \label{figure:GD_AGD_approximation}
  \caption{The first four iterates of GD and AGD, which form polynomials in $\eta$, the parameter of problem (\ref{opt:approximation_toy_problem}), are compared to $1/\eta$ over $[1,4]$.}
\end{figure}

It turns our that each $w^{(k)}(\eta)$ forms a univariate polynomial whose degree is at most $k$. Furthermore,  since $f_\eta(w)$ are $L$-smooth $\mu$-strongly convex for any $\eta\in[\mu,L]$, standard convergence analysis for GD (e.g., \cite{nesterov2004introductory}, Theorem 2.1.14)  guarantees that $|w^{(k)}(\eta) - w^*(\eta)|\le (1-2/(1+\kappa))^{\frac{k}{2}} |w^*(\eta)|$, 
where $\kappa$  denotes the condition number. Substituting \eqref{eq:gd_iterations} for $w^{(k)}(\eta)$ yields
\begin{align*}
	\max_{\eta\in[\mu,L]}\left|\frac{1}{L}\sum_{i=0}^{k-1} (-1)^{i} \binom{k}{i+1} (\eta/L)^i - 1/\eta\right|
	&\le \frac{1}{\mu}\circpar{1-\frac{2}{1+\kappa}}^{\frac{k}{2}}.
\end{align*}
Thus, we see that the faster the convergence rate of a given optimization algorithm is, the better the induced sequence of polynomials $(w^{(k)}(\eta))_{k\ge0}$ approximate $1/\eta$ w.r.t. the maximum norm $\| \cdot\|_{L_{\infty}([\mu,L])}$ over $[\mu,L]$. In Fig. 2
, we compare the first 4 polynomials induced by GD and AGD. Not surprisingly, AGD polynomials approximates $1/\eta$ better than those of GD. 

Now, one may ask, assuming that iterates of a given optimization algorithm $\cA$ for (\ref{opt:approximation_toy_problem}) can be expressed as polynomials $s_k(\eta)$  whose degree does not exceed the iteration number, just how fast can these iterates converge to the minimizer? Since the convergence rate is bounded from below by $\| s_k(\eta)- 1/\eta\|_{L_{\infty}([\mu,L])}$, we may address the following question instead:
\begin{align} \label{opt:elementary_problem}
	\min_{s(\eta)\in\cP_k}  \| s(\eta)- 1/\eta\|_{L_{\infty}([\mu,L])},
\end{align}
where $\cP_k$ denotes the set of univariate polynomials whose degree does not exceed $k$. Problem (\ref{opt:elementary_problem}) and other related settings are main topics of study in approximation theory. Accordingly, our technique for proving lower bounds makes an extensive use of tools borrowed from this area. Specifically, in a paper from 1899 \cite{tchebyshev1899polynomials} Chebyshev showed that
\begin{align} \label{ineq:basic_chebyshev}
	 \min_{s(\eta)\in\cP_k}  \left\| s(\eta)- \frac{1}{\eta-c}\right\|_{L_{\infty}([-1,1])}\ge\frac{(c-\sqrt{c^2-1})^k}{c^2-1},\quad c>1,
\end{align}
by which we derive the following theorem (see Appendix \ref{proof:thm:lb_toy} for a detailed proof). 
\begin{theorem} \label{thm:lb_toy}
The number of iterations required by $\cA$ to get an $\epsilon$-optimal solution is $\tilde{\Omega}(\sqrt{\kappa}\ln(1/\epsilon))$.
\end{theorem}
In the following sections, we apply oblivious CLI on various parameterized optimization problems so that the resulting iterates are polynomials in the problem parameters. We then apply arguments similar to the above

A similar reduction, from optimization problems to approximation problems, was used before in a few contexts to analyze the iteration complexity of deterministic CLIs (e.g., \cite[Section 3]{arjevani2016iteration}, see also Conjugate Gradient convergence analysis \cite{polyak1987introduction}). But, what if we allow random algorithms? should we expect the same iteration complexity? To answer this, we use Yao's minimax principle according to which the performance of a given stochastic optimization algorithm w.r.t. to its worst input are bounded from below by the performance of the best deterministic algorithm w.r.t. distributions over the input space. Thus, following a similar reduction one can show that the convergence rate of stochastic algorithms is bounded from below by
\begin{align} \label{opt:elementary_problem_L_1}
	\min_{s(\eta)\in\cP_k}  \int_{\mu}^L |s(\eta)- 1/\eta| \frac{1}{L-\mu} d\eta.
\end{align}
That is, a lower bound for the stochastic case can be attained by considering an approximation problem w.r.t. weighted $L_1$ with the uniform distribution over $[\mu,L]$. Other approximation problems considered in this work involve $L_2$-norm and different distributions. We provide a schematic description of our proof technique in~Scheme~2.1.
\begin{table}[H]
\vskip 0.15in
\begin{center}
\begin{small}
\begin{sc}
\begin{tabular}{llccr}\label{table:lower_bound_scheme}
\small \textbf{Scheme 2.1} &\small From Optimization Problems to Approximation Problems \\
\hline
\scriptsize \textbf{Given} &\scriptsize a class of functions $\cF$, a suitable oracle $\cO$\\ 
&\scriptsize and a sequence of sets of function $\cS_k$ over some parameters set $H$.\\
\scriptsize \textbf{Choose} &\scriptsize a subset of functions $\{f_\eta\in\cF|\eta\in H\}$, s.t.  $\bw^k(\eta)\in\cS_k$.\\
\scriptsize \textbf{Compute} &\scriptsize the minimizer $\bw^*(\eta)$ for any $f_\eta$\\
\scriptsize \textbf{Bound} &\scriptsize from below the best approximation for $\bw^*(\eta)$ w.r.t. $\cS_k$ \\ &
\scriptsize  and a norm $\|\cdot\|$, i.e., $\min\{\|\bs(\eta)-\bw^*(\eta)\|~|~ \bs(\eta)\in\cS_k\}$ 
\end{tabular}
\end{sc}
\end{small}
\end{center}
\vskip -0.1in
\end{table}

															\section{Lower Bound for Finite Sums Minimization Methods}

Having described our analytic approach, we now turn to present some concrete applications, starting with iteration complexity lower bounds in the context of FSM problems (\ref{opt:finite_sum}). In what follows, we derive a lower bound on the iteration complexity of oblivious (possibly stochastic) CLI algorithms equipped with first-order and coordinate-descent oracles for FSM. Strictly speaking, we focus on optimization algorithms equipped with both generalized first order oracle,
\begin{align} \label{oracle:first_order}
	\cO(\bw;A,B,C,j) &= A\nabla f_j(\bw) + B\bw+C,\quad A,B,C\in\reals^{d\times d}, j\in[n], 
\end{align}
and steepest coordinate-descent oracle
\begin{align} \label{oracle:fsm_cd}
	&\cO(\bw;i,j) &= \bw + t^*\be_i,\quad  t^*\in	\argmin_{t\in\reals} f_j(w_1,\dots,w_{i-1},w_i+t,w_{i+1},\dots,w_d), j\in[n],
\end{align}
where $\be_i$ denotes the $i$'th unit vector. We remark that coordinate-descent steps w.r.t. partial gradients can be implemented using (\ref{oracle:first_order}) by setting $A$ to be some principal minor of the unit matrix.It should be further noted that our results below hold for scenarios where the optimization algorithm is free to call a different oracle at different iterations. 

First, we sketch the proof of the lower bound for deterministic oblivious CLIs. Following Scheme 2.1, we restrict our attention to a parameterized subset of problems. We assume\footnote{Clearly, in order to derive a lower bound for coordinate-descent algorithms, we must assume $d>1$. If only a first-order oracle is allowed, then the same lower bound as in \thmref{thm:finite_sum_lb} can be derived for $d=1$.} $d>1$ and denote by $\cH_{\text{FSM}}$ the set of all $(\eta_1,\dots,\eta_n)\in\reals^n$  such that all the entries equal $-(L-\mu)/2$, except for some $j\in[n]$, for which $\eta_j\in \left[-(L-\mu)/2,(L-\mu)/2\right] $. Now, given $\bbeta\coloneqq(\eta_1,\dots,\eta_n)\in \cH_{\text{FSM}} $ we define 
\begin{align}\label{problem:finite_sum}
	F_{\bbeta}(\bw)&\coloneqq \frac{1}{n}\sum_{i=1}^n \circpar{\frac{1}{2}\bw^\top Q_{\eta_i} \bw - \bq},\text{where}\\
		Q_{\eta_i}&\coloneqq\mymat{\frac{L+\mu}{2} & \eta_i\\ \eta_i&\frac{L+\mu}{2} \\ && \mu \\ &&& \ddots \\ &&&&\mu },
		~\bq\coloneqq\mymat{\frac{R\mu}{\sqrt2} \vspace{0.05cm}\\\frac{R\mu}{\sqrt2}\\0\\\vdots\\0}.\nonumber
\end{align}
It is easy to verify that the minimizers of (\ref{problem:finite_sum}) are
\begin{align}\label{eq:finite_sum_minimizers}
	\bw^*(\bbeta) 
	=\circpar{\frac{R\mu}{\sqrt2\circpar{\frac{L+\mu}{2}+\frac{1}{n}\sum_{i=1}^n \eta_i } } ,
	          \frac{R\mu}{\sqrt2\circpar{\frac{L+\mu}{2}+\frac{1}{n}\sum_{i=1}^n \eta_i } },
						0,\dots,0}^\top.
\end{align} 
We would like to show that the coordinates of the iterates of deterministic oblivious CLIs, which minimize $F_{\bbeta}$ using first-order and coordinate-descent oracles, form multivariate polynomials in $\bbeta$ of total degrees (the maximal sum of powers over all the terms) which does not exceed the iteration number. Indeed, if the coordinates of $\bw_i^{(k)}(\bbeta)$ are multivariate polynomial in $\bbeta$ of total degree at most $k$, then the coordinates of the vectors returned by both oracles
\begin{align} \label{eq:oracles_finite_sum}
&\text{First-order oracle:}~\cO(\bw_j^{(k)};A,B,C,j) = A(Q_{\eta_j} \bw_i^{(k)} - \bq) + B\bw_i^{(k)}+C,\\
&\text{Coordinate-descent oracle:}~\cO(\bw_j^{(k)};i,j) = 
\circpar{I- (1/(Q_{\eta_j})_{ii})\be_i(Q_{\eta_j})_{i,*} }\bw_i^{(k)} - q_i/(Q_{\eta_j})_{ii}\be_i,\nonumber
\end{align}
are multivariate polynomials of total degree of at most $k+1$, as all the parameters ($A,B,C,i$ and $j$) do not depend on $\bbeta$ (due to obliviosity) and the rest of the terms ($Q_{\eta_j},\bq,I,1/(Q_{\eta_j})_{ii},(Q_{\eta_j})_{i,*},\be_i$ and $q_i$) are either linear in $\eta_j$ or constants. Now, since the next iterates are generated simply by summing up all the oracle answers, they also form multivariate polynomials of total degree of at most $k+1$. Thus, denoting the first coordinate of $\bw_1^{(k)}(\bbeta)$ by $s(\bbeta)$ and using \ineqref{ineq:basic_chebyshev}, we get the following bound 
\begin{align}  
	\max_{\bbeta\in \cH_{\text{FSM}}}
	\|\bw_1^{(k)}(\bbeta)-\bw^*(\bbeta)\|
	&\ge 
	\left\|s(\bbeta) - \frac{R\mu}{\sqrt2\circpar{\frac{L+\mu}{2}+\frac{1}{n}\sum_{i=1}^n \eta_i } }  \right\|_{L^\infty([\mu,L])} \label{ineq:convergence_sketch}\\
	&\ge 
	\Omega(1)\circpar{\frac{  \sqrt{\frac{\kappa-1}{n} +1} -1 }{  \sqrt{\frac{\kappa-1}{n} +1} +1 } }^{k/n}\!\!, \label{ineq:convergence_sketch_lb}
\end{align}
where $\Omega(1)$ designates a constant which does not depend on $k$ (but may depend on the problem parameters). Lastly, this implies that for any deterministic oblivious CLI and any iteration number, there exists some $\bbeta\in\cH_{\text{FSM}}$ such that the convergence rate of the algorithm, when applied on  $F_{\bbeta}$, is bounded from below by \ineqref{ineq:convergence_sketch_lb}. We note that, as opposed to other related lower bounds, e.g., \cite{lan2015optimal}, our proof is non-constructive. As discussed in subsection \ref{section:approximation}, this type of analysis can be extended to stochastic algorithms by considering (\ref{ineq:convergence_sketch}) w.r.t. other norms such as weighted $L_1$-norm. We now arrive at the following theorem whose proof, including the corresponding logarithmic factors and constants, can be found in Appendix \ref{proof:thm:finite_sum_lb}.
\begin{theorem}\label{thm:finite_sum_lb}
The iteration complexity of oblivious (possibly stochastic) CLIs for FSM (\ref{opt:finite_sum}) equipped with first-order (\ref{oracle:first_order}) and coordinate-descent oracles (\ref{oracle:fsm_cd}), is bounded from below by 
\begin{align*}
\tilde\Omega(n+ \sqrt{n(\kappa-1)}\ln(1/\epsilon)).
\end{align*}
\end{theorem}
The lower bound stated in \thmref{thm:finite_sum_lb} is tight and is attained by, e.g., SAG combined with an acceleration scheme (e.g., \cite{lin2015universal}). Moreover, as mentioned earlier, our lower bound does not depend on the problem dimension (or equivalently, holds for any number of iterations, regardless of $d$ and $n$), and covers coordinate descent methods with stochastic or deterministic coordinate schedule (in the special case where $n=1$, this gives a lower bound for minimizing smooth and strongly convex functions by performing steepest coordinate descent steps). Also, our bound implies that using mini-batches for tackling FSM does not reduce the overall iteration complexity. Lastly, it is noteworthy that the $n$ term in the lower bound above holds for any algorithm accompanied with an \emph{incremental oracle}, which grants access to at most one individual function each time. 

We also derive a nearly-optimal lower bound for smooth non-strongly convex functions for the more restricted setting of  $n=1$ and first-order oracle. The parameterized subset of functions we use (see Scheme 2.1)  is
$g_\eta(\bx)\coloneqq \frac\eta2 \norm{\bx}^2 - R\eta\be_1^\top\bx,\quad\eta\in(0,L]$.
The corresponding minimizer (as a function of $\eta$) is $\bx^{*}(\eta)=R\be_1$, and in this case we seek to approximate it w.r.t. $L_2$-norm using $k$-degree univariate polynomials whose constant term vanishes. The resulting bound is dimension-free and improves upon other bounds for this setting (e.g. \cite{arjevani2016iteration}) in that it applies to deterministic algorithms, as well as to stochastic algorithms (see \ref{proof:thm:smooth_lb} for proof).
\begin{theorem}\label{thm:smooth_lb}
The iteration complexity of any oblivious (possibly stochastic) CLI for $L$-smooth convex functions equipped with a first-order oracle, is bounded from below by 
\begin{align*}
\Omega\circpar{\circpar{L(\delta-2)/\epsilon}^{1/\delta}},~\delta\in(2,4).
\end{align*}
\end{theorem}											

							\section{Lower Bound for Dual Regularized Loss Minimization with Linear Predictors}\label{section:lb_dual_rlm}
The form of functions (\ref{problem:finite_sum})  discussed in the previous section does not readily adapt to general RLM problems with linear predictors, i.e.,
\begin{align} \label{opt:rlm}
	\min_{\bw\in\reals^d} P(\bw)&\coloneqq\frac{1}{n}\sum_{i=1}^n \phi_i(\inner{\bx_i,\bw}) + \frac{\lambda}{2}\norm{\bw}^2,
\end{align} 
where the loss functions $\phi_i$ are $L$-smooth and convex, the samples $\bx_1,\dots,\bx_n$ are $d$-dimensional vectors in $\reals^d$ and $\lambda$ is some positive constant. Thus, dual methods which exploit the added structure of this setting
through the \emph{dual problem} \cite{shalev2013stochastic}, 
\begin{align} \label{opt:dual_rlm}
	\min_{\balpha\in\reals^n} D(\balpha)=\frac{1}{n}\sum_{i=1}^n \phi^*_i(-\alpha_i) + \frac{\lambda}{2} \left\|\frac{1}{\lambda n} \sum_{i=1}^n \bx_i \alpha_i \right\|^2,
\end{align}
such as SDCA and accelerated proximal SDCA, are not covered by \thmref{thm:finite_sum_lb}. Accordingly, in this section, we address the iteration complexity of oblivious (possibly stochastic) CLI algorithms equipped with dual RLM oracles:
\begin{align} \label{oracle:dual_rlm}
&\cO(\balpha;t,j) = \balpha + t\nabla_j D(\balpha)\be_j, \quad t\in\reals,j\in[n],\\
&\cO(\balpha;j) = \balpha + t^*\be_j, \quad t^*=	\argmin_{t\in\reals} D(\alpha_1,\dots,\alpha_{j-1},\alpha_j +t,\alpha_{j+1},\dots,\alpha_d), j\in[n],\nonumber
\end{align}

Following Scheme 2.1, we first describe the relevant parametrized subset of RLM problems. 
For the sake of simplicity, we assume that $n$ is even (the proof for odd $n$ holds mutatis mutandis). We denote by $\cH_{\text{RLM}}$ the set of all $(\psi_1,\dots,\psi_{n/2})\in\reals^{n/2}$ such that all entries are $0$, except for some $j\in[n/2]$, for which $\psi_j\in \left[-\pi/2,\pi/2\right] $. Now, given ${\bpsi}\in\cH_{\text{RLM}}$, we set $P_{\bpsi}$ (defined in \ref{opt:rlm}) as follows 
\begin{align*}
	\phi_i(w)=\frac{1}{2}( w+1)^2,\quad  
	\bx_{\bpsi,i}&= \begin{cases}
	\cos(\psi_{(i+1)/2})\be_i +\sin(\psi_{(i+1)/2})\be_{i+1} 
	&i \text{ is odd}\\
	\be_i & \text{o.w.}
	\end{cases}.
\end{align*}
We state below the corresponding lower bound, whose proof, including logarithmic factors and constants, can be found in Appendix \ref{proof:thm:rlm_lb}. 
\begin{theorem}\label{thm:rlm_lb}
The iteration complexity of oblivious (possibly stochastic) CLIs for RLM (\ref{opt:rlm}) equipped with dual RLM oracles (\ref{oracle:dual_rlm}) is bounded from below by 
\begin{align*}
\tilde\Omega(n+\sqrt{nL/\lambda}\ln(1/\epsilon)).
\end{align*}
\end{theorem}
This bound is tight w.r.t. the class of oblivious CLIs and is attained by accelerated proximal SDCA. 
As mentioned earlier, a tighter lower bound of $\tilde\Omega((n+1/\lambda)\ln(1/\epsilon))$ is known for SDCA \cite{arjevani2015lower}, suggesting that a tighter bound might hold for the more restricted set of stationary CLIs (for which the oracle parameters remain fixed throughout the optimization process).

\newpage

\bibliographystyle{plain}
\bibliography{../../../mybib}


\newpage
																								\appendix

\section{Proofs}

\subsection{Proof of \thmref{thm:lb_toy}}	\label{proof:thm:lb_toy}
\begin{proof} 
According to the way $\cA$ generates iterates, we have
\begin{align*}
	|x^{(k)}(\eta) - x^*(\eta)|&= |s_k(\eta)-1/\eta|,\quad \eta\in[\mu,L]
\end{align*}
for some polynomial $s_k(\eta)$ of degree at most $k$. By \lemref{lem:max_norm_general}, we have 
\begin{align*}
		\min_{s(\eta)\in\cP_k}\left\|s(\eta)-\frac{1}{\eta} \right\|_{L_{\infty}([\mu,L])} \ge\frac{L-\mu}{ 2 L\mu }\circpar{\frac{\sqrt{\kappa}-1}{\sqrt{\kappa}+1}}^k,
\end{align*}
where $\kappa=L/\mu$. Thus,
\begin{align*}
	|x^{(k)}(\eta) - x^*(\eta)|&\ge 	\min_{s(\eta)\in\cP_k}\left\|s(\eta)-\frac{1}{\eta} \right\|_{L_{\infty}([\mu,L])} \ge\frac{L-\mu}{ 2 L\mu }\circpar{\frac{\sqrt{\kappa}-1}{\sqrt{\kappa}+1}}^k\ge|x^*(\eta)|\frac{L-\mu}{ 2 L }\circpar{\frac{\sqrt{\kappa}-1}{\sqrt{\kappa}+1}}^k.
\end{align*}
Now, since $f_\eta$ is $\mu$-strongly convex, we have,
\begin{align*}
	f(x^{(k)}(\eta)) - f(x^*(\eta))|&\ge 	
	\frac{\mu}{2}|x^{(k)}(\eta) - x^*(\eta)|^2\\&\ge \frac{\mu}{2} \circpar{|x^*(\eta)|\frac{L-\mu}{ 2 L }\circpar{\frac{\sqrt{\kappa}-1}{\sqrt{\kappa}+1}}^k}^2\\&=
	 \frac{\mu}{2} \circpar{|x^*(\eta)|\frac{L-\mu}{ 2 L }}^2\circpar{\frac{\sqrt{\kappa}-1}{\sqrt{\kappa}+1}}^{2k}.
\end{align*}
Hence, by \lemref{lem:technical2tag}, the minimal number of iterations required to get an $\epsilon$-optimal solution is at least
\begin{align*}
	\frac{1}{4}{\sqrt{\kappa-1}} \circpar{\ln\frac\mu2+ 2\ln\circpar{|x^*(\eta)|\frac{L-\mu}{ 2 L }} + \ln(1/\epsilon ) }.
\end{align*}
\end{proof}

								\subsection{Proof of \thmref{thm:finite_sum_lb} - Finite Sums}	\label{proof:thm:finite_sum_lb}
When dealing with multivariate polynomials it is convenient to define multi-indices $\bi=(i_1,\dots,i_n)\in \bN_0^n$ , where $\bN_0^n$ is the set of all $n$-tuples of non-negative integers. In addition, with a slight abuse of notation, we define 
\begin{align}\label{def:multivariate_polynomials}
	\cP_k^n\coloneqq \spn{\eta^{\bi}~|~~\bi\in\bN_0^n,~|\bi|\le k},
\end{align} 
where we put $\eta^\bi=\eta_1^{i_1} \cdots\eta_n^{i_n}$ and $|\bi|=i_1+\cdots+i_n$. In words, $\cP_k^n$ is the set of all multivariate polynomials over $n$ indeterminates whose total degree (the maximal sum of the degrees over all terms) is less than or equal to $k$. Lastly, given $\bs(\bbeta)\in\cP_k^n$ we define
\begin{align*}
	s_i(\eta_i) &\coloneqq \bs\circpar{-\frac{L-\mu}{2},\dots,-\frac{L-\mu}{2},\underbrace{\eta_i}_{i\text{'th entry}},-\frac{L-\mu}{2},\dots,-\frac{L-\mu}{2}}.
\end{align*}
This notation will come in handy in the main proof.
								


The lemma below describes the functional form assumed by iterates produced by oblivious CLIs.
\begin{lemma} \label{lem:pcli_coefficients_polynomials_finite_sum}
When applied on (\ref{problem:finite_sum}) with suitable first-order and coordinate-descent oracles (as defined in \ref{eq:oracles_finite_sum}), the coordinates of iterates produced by oblivious stochastic CLIs form multivariate polynomials in $\bbeta$ with random real coefficients whose total degree does not exceed the iteration number.
\end{lemma}
\begin{proof}
Let $\cA$ be a oblivious stochastic CLI, and suppose we apply $\cA$ on the class of problems (\ref{problem:finite_sum}) parameterized by $\bbeta$, using both first-order and coordinate-descent oracles as defined in \ref{eq:oracles_finite_sum}. We use mathematical induction to show that for any $k=0,1,\dots$, the coordinate of the $k$'th iterate produced by such process can be expressed as a distribution over multivariate polynomials in $\bbeta$ of degree at most $k$.

As the first iterate $\bw^{(0)}_i$ is allowed to depend only on $L,\mu$ and $n$, the base case is trivial. That is, the coordinates of $\bw_i^{(0)}$ form distributions over $\reals=\cP^n_0$ which do not depend on $\bbeta$.

For the inductive step, assume that any coordinate of $\bw^{(k)}_i(\bbeta)$ can be expressed as a distribution over $\cP^n_k$. It is easy to see that for any $\bw_i^{(k)}(\bbeta)$, the answers of both oracles, 
\begin{align*} 
\begin{array}{lll}
\text{First-order oracle:}	&\cO(\bw_i^{(k)};A,B,C,j) = A(Q_{\eta_j} \bw_i^{(k)} - \bq) + B\bw_i^{(k)}+C,\\
\text{Coordinate-descent oracle:}&\cO(\bw_i^{(k)};i,j) = 
\circpar{I- (1/(Q_{\eta_j})_{ii})\be_i(Q_{\eta_j})_{i,*} }\bw_i^{(k)} - q_i/(Q_{\eta_j})_{ii}\be_i,
\end{array}
\end{align*}
form a distribution over $\cP^n_{k+1}$, as all the random quantities involved in the expressions ($A,B,C,i$ and $j$) do not depend on $\eta_1,\dots,\eta_n$ (due to obliviosity) and the rest of the terms ($I,Q_{\eta_j},1/(Q_{\eta_j})_{ii},(Q_{\eta_j})_{i,*},\be_i,q_i$ and $\bq$) are either linear in $\eta_j$ or constants. Lastly, $\bw_i^{(k+1)}$ are computed by simply summing up all the oracle answers, and as such, form again distributions over $\cP^n_{k+1}$.
\end{proof}
\begin{proof}[\thmref{thm:finite_sum_lb}]
Let $\cA$ be a oblivious stochastic CLI. By \lemref{lem:pcli_coefficients_polynomials_finite_sum} the first coordinate of $\bw^{(k)}_1(\bbeta)$ (the point returned by the algorithm at the $k$'th iteration) when applied on the class of problems (\ref{problem:finite_sum}) distributes according to some distribution $\cD$ over $\cP^n_k$. 
 Thus, by Yao principle, since each polynomial in $(\cP^n_k)^d$ represents a single deterministic algorithm, we have 
\begin{align} \label{ineq:proof:thm:finite_sum_lb_ineq1}
	\max_{\bbeta\in\cH}\bE_{\bw_1^{(k)}(\bbeta)\sim\cD}\|\bw_1^{(k)}(\bbeta)-\bw^*(\bbeta)\|\ge \min_{\bs(\bbeta)\in(\cP^n_k)^d} \bE_{\bbeta\sim\cU(\cH)} \|\bs(\bbeta)-\bw^*(\bbeta)\|
\end{align}
where $\cU(\cH)$ denotes a distribution over $\cH$ which corresponds to  first drawing $j\sim\cU([n])$ at random, and then setting the coordinates of $\bbeta$ as follows
\begin{align}
\begin{cases}
\eta_i\sim\cU([-(L-\mu)/2,(L-\mu)/2]&i=j\\
\eta_i=-\frac{L-\mu}{2},&i\neq j
\end{cases}.
\end{align}
Furthermore, it is easy to verify that the corresponding minimizers of (\ref{problem:finite_sum}) are
\begin{align}
	\bw^*(\eta_1,\dots,\eta_n) &= \circpar{\frac{1}{n}\sum_{i=1}^n Q_{\eta_i}}^{-1}\bq
	=\circpar{\frac{R\mu}{\sqrt2\circpar{\frac{L+\mu}{2}+\frac{1}{n}\sum_{i=1}^n \eta_i } } ,
	          \frac{R\mu}{\sqrt2\circpar{\frac{L+\mu}{2}+\frac{1}{n}\sum_{i=1}^n \eta_i } },
						0,\dots,0}^\top.
\end{align} 
We now have,
\begin{align}\label{ineq:proof:thm:finite_sum_lb_ineq2}
\min_{\bs(\bbeta)\in(\cP^n_k)^d} \bE_{\bbeta\sim\cU(\cH)} \|\bs(\bbeta)-\bw^*(\bbeta)\| &=
\min_{\bs(\bbeta)\in(\cP^n_k)^d} \bE_{i\sim\cU([n])} \bE_{\eta_i\sim\cU([-\frac{L-\mu}{2},\frac{L-\mu}{2}])} \|\bs(\bbeta)-\bw^*(\bbeta)\| \nonumber\\
&\ge
\frac{1}{n}\min_{s(\bbeta)\in\cP^n_k}\sum_{j=1}^n \bE_{\eta_j\sim\cU([-\frac{L-\mu}{2},\frac{L-\mu}{2}])}  \left|s_j(\eta_j)-\frac{R\mu}{\sqrt2(\frac{1}{n}\sum_{i=1}^n\eta_i) +\frac{L+\mu}{2}) }\right|\nonumber\\
&\ge
\frac{R\mu}{\sqrt2}\min_{s(\bbeta)\in\cP^n_k}\sum_{j=1}^n\bE_{\eta_j\sim\cU([-\frac{L-\mu}{2},\frac{L-\mu}{2}])}  \left|s_j(\eta_j)-\frac{1}{\eta_j - (n-1)\frac{L-\mu}{2} +n\frac{L+\mu}{2} }\right|\nonumber\\
&\ge
\frac{R\mu}{\sqrt2}\min_{s(\bbeta)\in\cP^n_k}\sum_{j=1}^n \int_{-\frac{L-\mu}{2}}^{\frac{L-\mu}{2}} \left|s_j(\eta_j)-\frac{1}{\eta_j - (n-1)\frac{L-\mu}{2} +n\frac{L+\mu}{2} }\right|\frac{1}{L-\mu} d\eta_j\nonumber\\
&\ge
\frac{R\mu}{\sqrt2(L-\mu)}\min_{s(\bbeta)\in\cP^n_k}\sum_{j=1}^n \int_{-\frac{L-\mu}{2}}^{\frac{L-\mu}{2}} \left|s_j(\eta_j)-\frac{1}{\eta_j - (n-1)\frac{L-\mu}{2} +n\frac{L+\mu}{2} }\right| d\eta_j
\end{align}
where the first inequality follows by focusing on the first coordinate of $\bs(\bbeta)-\bw^*(\bbeta)$. Now, set $\alpha=-(n-1)\frac{L-\mu}{2} +n\frac{L+\mu}{2}$  and note that
\begin{align*}
\sqrt\frac{2\alpha+L-\mu}{2\alpha+\mu-L}=
	\sqrt\frac{2(-(n-1)\frac{L-\mu}{2} +n\frac{L+\mu}{2})+L-\mu}{2(-(n-1)\frac{L-\mu}{2} +n\frac{L+\mu}{2})+\mu-L}&=
	\sqrt{\frac{\kappa-1}{ n} + 1}.
\end{align*}
Thus, by \lemref{appendix:lem:l1_norm_lb} (using the same value for $\alpha$ and noting that $\alpha>(L-\mu)/2$) yields 
\begin{align*}
	\int_{-\frac{L-\mu}{2}}^{\frac{L-\mu}{2}} \left|s_j(\eta_j)-\frac{1}{\eta_j - (n-1)\frac{L-\mu}{2} +n\frac{L+\mu}{2} }\right| d\eta_j&\ge\circpar{\frac{\sqrt{\frac{\kappa-1}{ n} + 1}-1}{	\sqrt{\frac{\kappa-1}{ n} + 1}+1}}^{k_j}.
\end{align*}
where $k_j$ denotes the degree of $s_j(\eta_j)$. Plugging in this into \ineqref{ineq:proof:thm:finite_sum_lb_ineq2} we get 
\begin{align*}
\max_{\bbeta\in\cH}\bE_{\bw_1^{(k)}(\bbeta)\sim\cD}\|\bw_1^{(k)}(\bbeta)-\bw^*(\bbeta)\|&\ge
\frac{nR\mu}{\sqrt2(L-\mu)}\min_{\bs(\bbeta)\in\cP^n_k}\frac{1}{n}\sum_{j=1}^n \circpar{\frac{\sqrt{\frac{\kappa-1}{ n} + 1}-1}{	\sqrt{\frac{\kappa-1}{ n} + 1}+1}}^{k_j}.
\end{align*}
Since $u\mapsto\rho^u$ is a decreasing and convex function for any $1>\rho>0$, we have
\begin{align*}
\frac{nR\mu}{\sqrt2(L-\mu)}\min_{\bs(\bbeta)\in\cP^n_k}\frac{1}{n}\sum_{j=1}^n \circpar{\frac{\sqrt{\frac{\kappa-1}{ n} + 1}-1}{	\sqrt{\frac{\kappa-1}{ n} + 1}+1}}^{k_j}&\ge
\frac{nR\mu}{\sqrt2(L-\mu)}\min_{\bs(\bbeta)\in\cP^n_k} \circpar{\frac{\sqrt{\frac{\kappa-1}{ n} + 1}-1}{	\sqrt{\frac{\kappa-1}{ n} + 1}+1}}^{\frac{1}{n}\sum_{j=1}^n k_j}\\&\ge
\frac{nR\mu}{\sqrt2(L-\mu)}\circpar{\frac{\sqrt{\frac{\kappa-1}{ n} + 1}-1}{	\sqrt{\frac{\kappa-1}{ n} + 1}+1}}^{k/n}
\end{align*}
where the last inequality is due to the fact that $\bs(\bbeta)\in\cP^n_k$ which implies that $\sum_{j=1}^n k_j \le k$.
Finally, we have,
\begin{align*}
\max_{\bbeta\in\cH}\bE_{\bw_1^{(k)}(\bbeta)\sim\cD}[ F_{\bbeta}(\bw_1^{(k)}(\bbeta))-F_{\bbeta}(\bw^*(\bbeta))]&\ge
	\max_{\bbeta\in\cH}\bE_{\bw_1^{(k)}(\bbeta)\sim\cD}\frac{\mu}{2}\|\bw_1^{(k)}(\bbeta)-\bw^*(\bbeta)\|^2\\&\ge
	\frac\mu2\circpar{\frac{nR\mu}{\sqrt2(L-\mu)}\circpar{\frac{\sqrt{\frac{\kappa-1}{ n} + 1}-1}{	\sqrt{\frac{\kappa-1}{ n} + 1}+1}}^{k/n}}^2\\&=
		\frac\mu2\circpar{\frac{nR\mu}{\sqrt2(L-\mu)}}^2\circpar{\frac{\sqrt{\frac{\kappa-1}{ n} + 1}-1}{	\sqrt{\frac{\kappa-1}{ n} + 1}+1}}^{2k/n}
\end{align*}
where the first inequality follows by the $\mu$-strong convexity of $F_{\bbeta}$ and the second inequality follows by Jensen inequality.
Using \lemref{lem:technical2tag}, we get that the iteration complexity of $\cA$ is at least 
\begin{align*}
		 \frac{1}{4}\circpar{\sqrt{n(\kappa-1) }} (\ln\frac\mu2+2\ln\frac{nR\mu}{\sqrt2(L-\mu)} + \ln(1/\epsilon ) ).
\end{align*}
This, together with \thmref{thm:non_trivial_accuracy_finite_sum} below, concludes the proof.

\end{proof}

We bound from below the number of iterations required to obtain a non-trivial accuracy. 
\begin{lemma}\label{lemma:n_iteration_finite_sum}
Let $j\in[n]$, let $\bbeta_{j,1}\in\cH$ be a parameters vector whose all coordinates are $-\frac{L-\mu}{2}$ and let $\bbeta_{j,2}\in\cH$ be a parameters vector whose all coordinates are $-\frac{L-\mu}{2}$, except for the $j$'th coordinate which we set to be $\frac{L-\mu}{2}$. If $\kappa>3$, then  
\begin{align*}
	\|\bw^*(\bbeta_1) -  \bw^*(\bbeta_2)\|\ge \frac{2R}{n+2}.
\end{align*} 
\end{lemma}
\begin{proof}
By \eqref{eq:finite_sum_minimizers} we have
\begin{align*}
	\|\bw*\circpar{\bbeta_1} -  \bw*\circpar{\bbeta_2}\|
	&=\sqrt2\left|\frac{R\mu}{\sqrt2\circpar{\frac{L+\mu}{2}+\frac{1}{n}\sum_{i=1}^n (\eta_1)_i } } - 
	\frac{R\mu}{\sqrt2\circpar{\frac{L+\mu}{2}+\frac{1}{n}\sum_{i=1}^n (\eta_2)_i } }\right|\\
	&=R\mu\left|\frac{1}{\frac{L+\mu}{2}-\frac{L-\mu}{2}}   - 
	\frac{1}{\frac{L+\mu}{2}-\frac{(n-1)(L-\mu)}{2n} +\frac{L-\mu}{2n}  }\right|\\
	&=R\mu\left|
	\frac{\frac{L+\mu}{2}-\frac{(n-1)(L-\mu)}{2n} +\frac{L-\mu}{2n}  - \frac{L+\mu}{2}+\frac{L-\mu}{2}}
	{	\circpar{\frac{L+\mu}{2}-\frac{L-\mu}{2}}	\circpar{\frac{L+\mu}{2}-\frac{(n-1)(L-\mu)}{2n} +\frac{L-\mu}{2n}  }}\right|\\
	&=R\left|
	\frac{-\frac{(n-1)(L-\mu)}{n} +\frac{L-\mu}{n}  +L-\mu}
	{	L+\mu-\frac{(n-1)(L-\mu)}{n} +\frac{L-\mu}{n}  }\right|\\
	&=2R\left|
	\frac{\frac{L-\mu}{n}  }
	{	L+\mu-\frac{(n-1)(L-\mu)}{n} +\frac{L-\mu}{n}  }\right|\\
	&=2R\left|
	\frac{1}
	{	n\frac{\kappa+1}{\kappa-1}-(n-1) +1  }\right|\\
	&=2R\left|
	\frac{1}
	{	n\frac{\kappa+1}{\kappa-1}-n +2  }\right|\\
	&\ge\frac{2R}{n+2},
\end{align*}
where the last inequality follows from $\kappa>3$.
\end{proof}
\begin{theorem}\label{thm:non_trivial_accuracy_finite_sum}
The iteration complexity of any stochastic optimization algorithm (not necessarily CLI) which gathers information on $F_{\bbeta}$ (with $\kappa>3$) only by means of incremental oracles, i.e., oracles which upon receiving query return an answer which depends on not more than one individual function, is at least $n$.
\end{theorem}
\begin{proof}
Let $\cA$ be a stochastic optimization algorithm. According to Yao's principle, we can bound from below the $\epsilon$-optimality of $\cA$ after $k<n$ iterations by estimating the $\epsilon$-optimality of any deterministic algorithm w.r.t. to distribution $\cD(\cH)$ over $\cH$ defined by: draw $j\in[n]$ and set $\bbeta$ to be $\bbeta_{j,1}$ or $\bbeta_{j,2}$ as defined in \lemref{lemma:n_iteration_finite_sum} w.p. $1/2$. Then,
\begin{align*}
	\max_{\{\bbeta_{j,i}|j\in[n],i\in[2]\}} &\bE_{\cA}[ F_{\bbeta_{j,i}}(\bw^{(k)}\circpar{\bbeta_{j,i}}) - F_{\bbeta_{j,i}}(\bw^{*}\circpar{\bbeta_{j,i}}) ] \\&\ge 
	\min_{\text{deterministic algorithms}}\bE_{_{\bbeta}\sim\cD(\cH)}[ F_{\bbeta_{j,i}}(\bw^{(k)}\circpar{\bbeta_{j,i}}) - F_{\bbeta_{j,i}}(\bw^{*}\circpar{\bbeta_{j,i}}) ]\\
	&\ge 
	\min_{\text{deterministic algorithms}}\bE_{_{\bbeta}\sim\cD(\cH)} 
	\frac{\mu}{2}\| \bw^{(k)}\circpar{\bbeta_{j,i}}) - \bw^{*}\circpar{\bbeta_{j,i}} \|^2\\
	&\ge 
	\frac\mu2\min_{\text{deterministic algorithms}}\circpar{\bE_{_{\bbeta}\sim\cD(\cH)} 
	\| \bw^{(k)}\circpar{\bbeta_{j,i}}) - \bw^{*}\circpar{\bbeta_{j,i}} \|}^2\\
&\ge \frac\mu2\circpar{\frac{R}{2n(n+2)}}^2,
\end{align*}
where the last inequality follows from \lemref{lemma:n_iteration_finite_sum}. Thus, for sufficiently small $\epsilon$, one must perform at least $n$ iterations in order to obtain an $\epsilon$-optimal solution.
\end{proof}

							\subsection{Proof of \thmref{thm:smooth_lb} - Smooth Functions}	\label{proof:thm:smooth_lb}
The following notation
\begin{align}\label{def:polynomials_zero}
	\overline{\cP}_k\coloneqq \{p\in\cP_k|p(0)=0\}
\end{align} 
will come in handy in subsequent proofs.

\begin{lemma} \label{lem:pcli_coefficients_polynomials_smooth}
When applied on 
\begin{align}\label{problem:smooth}
g_\eta(\bx)\coloneqq \frac\eta2 \norm{\bx}^2 - R\eta\be_1^\top\bx,\quad\eta\in(0,L]
\end{align}
 with a first-order oracle (as defined in \ref{oracle:first_order} with $n=1$), the coordinates of iterates produced by oblivious stochastic CLIs whose is initialization iterate is $\bx_i^{(0)}=0$ form polynomials in $\eta$ with random real coefficients which vanishes at $\eta=0$ and whose degree does not exceed the iteration number.
\end{lemma}

\begin{proof}
Let $\cA$ be a oblivious stochastic CLI, and suppose we apply $\cA$ on the class of problems (\ref{problem:smooth}) parameterized by $\eta$, using a first-order. We use mathematical induction to show that for any $k=0,1,\dots$, the coordinate of the $k$'th iterate produced by such process can be expressed as a distribution over $\overline{\cP}_k$.

As the first iterate $\bx^{(0)}_i$ is assumed to be zero, the base case is trivial. For the inductive step, assume that any coordinate of $\bx^{(k)}_i$ can be expressed as a distribution over $\overline{\cP}_k$. It is easy to see that for any $\bx_i^{(k)}$, the answers of the first-order oracle,
\begin{align*} 
\text{First-order oracle:}	\quad&\cO(\bx_i^{(k)};A,B,C) = A(\eta \bx_i^{(k)} - R\eta\be_1) + B\bx_i^{(k)}+C,
\end{align*}
form a distribution over $\cP_{k+1}^0$, as the random quantities involved in the expressions ($A,B$ and $C$) do no depend on $\eta$ (due to obliviosity) and the rest of the terms ($\eta$ and $R\eta\be_i$) are homogenous in $\eta$. Lastly, $\bx_i^{(k+1)}$ are computed by simply summing up all the oracle answers, and as such, form again distributions over $\cP_{k+1}^0$.

\end{proof}

\begin{proof}[\thmref{thm:smooth_lb}]
Let $\cN$ be a oblivious stochastic CLI and let $\alpha\in(-1,0)$. Our derivation of lower bounds for stochastic CLIs is established via Yao principle. Fix some $k\in\{0,1,\dots\}$. By \lemref{lem:pcli_coefficients_polynomials_smooth}, $\bx^{(k)}_1(\eta)$ distributes according to some distribution $\cD$ over $(\overline{\cP}_k)^d$. Thus, by Yao principle, since each polynomial in $(\overline{\cP}_k)^d$ represents a single deterministic algorithm, we have 
\begin{align*}
	\max_{\eta\in[0,L]}\bE_{\bx_1^{(k)}(\eta)\sim\cD} g_\eta(\bx_1^{(k)}(\eta))-g_\eta(\bx^*(\eta))\ge \min_{\bs(\eta)\in(\overline{\cP}_k)^d} \bE_{\eta\sim\cE([0,L])} g_\eta(\bs(\eta))-g_\eta(\bx^*(\eta))
\end{align*}
where $\cE([0,L],\alpha)$ (abbr. $\cE$) denotes a distribution over $(0,L]$ with a probability density function $$p_\cE(\eta)=\frac{(\alpha+1)\eta^\alpha}{L^{\alpha+1}}.$$ We have,
\begin{align*}
\min_{\bs(\eta)\in(\overline{\cP}_k)^d} \bE_{\eta\sim\cE} [g_\eta(\bs(\eta))-g_\eta(\bx^*(\eta))] &\ge
\min_{s(\eta)\in\overline{\cP}_k} \bE_{\eta\sim\cE} \left[\eta\|\bs(\eta)-\bx^*(\eta)\|^2 \right]\\
&\ge\min_{s(\eta)\in\overline{\cP}_k} \bE_{\eta\sim\cE} \left[\eta(s(\eta)-R )^2\right]\\
&=R^2\min_{s(\eta)\in\overline{\cP}_k} \bE_{\eta\sim\cE} \left[\eta(s(\eta)-1 )^2\right]\\
&=\frac{R^2(\alpha+1)}{L^{\alpha+1}}\min_{s(\eta)\in\overline{\cP}_k} \int_{0}^L \eta(s(\eta)-1)^2 \eta^\alpha d\eta\\
&=\frac{R^2(\alpha+1)}{L^{\alpha+1}}\min_{s(\eta)\in\overline{\cP}_k} \int_{0}^1 L\eta(s(L\eta)-1)^2 (L\eta)^\alpha L~ d\eta\\
&=LR^2(\alpha+1)\min_{s(\eta)\in\overline{\cP}_k} \int_{0}^1 \eta(s(\eta)-1)^2 \eta^\alpha ~ d\eta
\end{align*}
where the first inequality follows by the fact that $h_\eta$ is $\eta$-strongly convex and the second inequality follows by focusing on the first coordinate of $\bs(\eta)-\bx^*(\eta)$. Invoking \lemref{appendix:lem:l2_norm_lb} yields
	\begin{align*}
	LR^2(\alpha+1)\min_{s(\eta)\in\overline{\cP}_k} \int_{0}^1 \eta(s(\eta)-1)^2 \eta^\alpha ~ d\eta &= LR^2(\alpha+1)\min_{s(\eta)\in\cP_{k-1}} \int_{0}^1 \eta(s(\eta)\eta-1)^2 \eta^\alpha d\eta,\\&\ge 
	\frac{LR^2(\alpha+1)}{e^{2} (k+2)^{2(\alpha+1)+2}}.
\end{align*}
Thus, in this case the iteration complexity is bound from below by
\begin{align*}
	\sqrt[2(\alpha+1)+2]{\frac{LR^2(\alpha+1)}{e^2 \epsilon}} -2.
\end{align*}

\end{proof}

								\subsection{Proof of \thmref{thm:rlm_lb} - Regularized Empirical Loss Minimization}	\label{proof:thm:rlm_lb}
For ease of presentation, we assume that $\|\bx_i\|\le1$, $\phi_i$ take non-negative values and $\phi_i(0)\le1$. Furthermore, throughout the proof we assume that $n$ is even and that $L=1$ (the proof for odd $n$ and general $L>0$ holds mutatis mutandis). First, we give an explicit definition of the parametrized set of functions we will be focusing on, as well as the oracles under which our bounds hold. We denote by $\cH$ the set of all $(\psi_1,\dots,\psi_{n/2})\in\reals^{n/2}$ such that all the entries are $0$, except for some $j\in[n/2]$, for which $\psi_j\in \left[-\pi/2,\pi/2\right] $. Now, given $\bpsi\in\cH$, we set 
\begin{align*}
	\phi_i(w)&=\frac{1}{2}( w+1)^2 \implies
	\phi_i^*(u)= \frac{1}{2}u^2 - u\\
	\bx_{\bpsi,i}&= \begin{cases}
	\cos(\psi_{(i+1)/2})\be_i +\sin(\psi_{(i+1)/2})\be_{i+1} 
	&i \text{ is odd}\\
	\be_i & \text{o.w.}
	\end{cases}.
\end{align*}
In which case, the corresponding dual is:
\begin{align} \label{problem:rlm}
	D_{\bpsi}(\balpha)=\frac{1}{2n}\norm{\balpha}^2  -\frac{1}{n}\bones^\top\balpha+ \frac{1}{2\lambda n^2}\norm{  X_{\bpsi} \alpha_i }^2
\end{align}
where 
\begin{align*}
	X_{\bpsi} &\coloneqq \circpar{\bx_{\bpsi,1},\dots,\bx_{\bpsi,n}}.
\end{align*}
Equivalently 
\begin{align*}
	D_{\bpsi}=\frac{1}{2}\balpha^\top \circpar{ \frac{1}{n} I + \frac{1}{\lambda n^2}X_{\bpsi}^\top X_{\bpsi}}\balpha - \frac{1}{n}\bones^\top\balpha
\end{align*}
Note that 
\begin{align*}
	Q_{\bpsi}\coloneqq \frac{1}{n}I+\frac{1}{\lambda n^2}X_{\bpsi}^\top X_{\bpsi}&= \frac1n\mymat{
	1+\frac{1}{\lambda n} & \frac{1}{\lambda n}\sin(\psi_1) \vspace{0.05cm}\\ 
	\frac{1}{\lambda n}\sin(\psi_1) & 1+\frac{1}{\lambda n}\\ 
	&&1+\frac{1}{\lambda n} & \frac{1}{\lambda n}\sin(\psi_2) \vspace{0.05cm}\\ 
	&&\frac{1}{\lambda n}\sin(\psi_2) & 1+\frac{1}{\lambda n}\\
	&&&&\ddots
	}.
\end{align*}
Note that, all the eigenvalues of $Q_{\bpsi}$ are bigger than 1. Therefore, $D_{\bpsi}$ is 1-strongly convex. 
We assume that the oracles at the algorithms' disposal are the dual RLM oracles defined in (\ref{oracle:dual_rlm}),

Lastly, we will need the following definitions 
\begin{align}
	\cP^n_{k,d}(\eta_1,\eta_2,\dots,\eta_n)&\coloneqq \left\{ \mymat{p_1(\eta_1,\eta_2,\dots,\eta_n)\\\vdots\\p_d(\eta_1,\eta_2,\dots,\eta_n)} \middle|~p_1,\dots,p_d\in \cP_k^n,\quad \partial p_1+\cdots+\partial p_d\le k  \right\}\\
	\cQ^n_{k,d}(\psi_1,\psi_2,\dots,\psi_n)&\coloneqq \left\{ \mymat{p_1(\sin\psi_1,\sin\psi_2,\dots,\sin\psi_n)\\\vdots \\p_d(\sin\psi_1,\sin\psi_2,\dots,\sin\psi_n)} \middle|~p_1,\dots,p_d\in \cP_{k,d}^n \right\}
\end{align} 
to ease notation in subsequent proofs (where $\partial p$ denotes the total degree of $p$ and $\cP_k^n$ is defined in (\ref{def:multivariate_polynomials})). Thus, $\cQ_k^n$ contains $d$-dimensional vectors whose entries are $n$-multivariate polynomials expressions in $\sin\psi_1,\dots,\sin\psi_n$, such that the sum of the degrees of the $d$-polynomials does not exceed $k$. In addition, given $\bt(\bpsi)\in\cQ_{k,d}^n$ we define
\begin{align*}
	t_i(\psi_i) &\coloneqq \bt\circpar{0,\dots,0,\underbrace{\psi_i}_{i\text{'th entry}},0,\dots,0}, \quad\forall i\in[d].
\end{align*}

As usual, we start by stating the functional form assumed by iterates produced by this sort of optimization algorithms.
\begin{lemma} \label{lem:pcli_coefficients_polynomials_rlm}
When applied on (\ref{problem:rlm}) with a dual RLM oracle (as defined in \ref{oracle:dual_rlm}), the coordinates of iterates produced by oblivious stochastic CLIs form $n$ multivariate polynomials  expressions in $\sin\psi_1,\dots,\sin\psi_{n/2}$ with random coefficients, such that the sum of the degrees of these polynomials does not exceed the iteration number.
\end{lemma}
\begin{proof}
Let $\cA$ be a oblivious stochastic CLI, and suppose we apply $\cA$ on the class of problems (\ref{problem:rlm}) parameterized by $\bpsi$, using dual RLM oracles as defined in \ref{oracle:dual_rlm}. We use mathematical induction to show that for any $k=0,1,\dots$, the coordinate of the $k$'th iterate produced by such process can be expressed as a distribution over polynomial expressions in $\sin\psi_1,\dots,\sin\psi_{n/2}$ whose sum of degrees is less than or equal $k$.

As the first iterate $\balpha^{(0)}_i$ is allowed to depend only on $n$ and $\lambda$, the base case is trivial. That is, $\balpha_i^{(0)}$ forms a distribution over $\reals^n=\cQ^{n/2}_{0,n}$ which does not depend on $\sin\psi_1,\dots,\sin\psi_{n/2}$.

For the inductive step, assume that $\balpha^{(k)}_i$ can be expressed as a distribution over $\cQ^{n/2}_{k,n}$. It is easy to see that for any $\balpha_i^{(k)}$, the answer of the dual RLM oracle 
\begin{align*} 
&\cO(\balpha_i^{(k)};t,\ell) = \balpha + t\be_\ell^\top(Q_{\bpsi}\balpha_i^{(k)} -\frac{1}{n}\bones)\be_\ell, \quad t\in\reals,j\in[n],\\
&\cO(\balpha_i^{(k)};\ell) = \circpar{I-\frac{1}{ (Q_{\bpsi})_{\ell\ell} } \be_\ell(Q_{\bpsi})_{\ell,*}}\balpha^{(k)} +\frac{1}{n (Q_{\bpsi})_{\ell\ell}}\be_\ell\\
\end{align*}
are distributions over $\cQ^{n/2}_{k+1,n}$, as the only random quantity involved in the expressions $t,\ell$ does not depend on $\bpsi$ (due to obliviosity), the only linear factor in $\sin\psi_\ell$ (i.e., $\be_\ell^\top(Q_{\bpsi}\balpha -\frac{1}{n}\bones)\be_\ell,\be_\ell(Q_{j,\eta})_{\ell,*}$)
'touches' $\balpha^{(k)}_i$ at exactly one entry and the rest of the terms ($1/n\bones,I,1/(Q_{j,\eta})_{\ell\ell}$ and $n$) are constants (w.r.t. $\sin\psi_\ell$ ). Lastly, $\balpha_i^{(k+1)}$ are computed by simply summing up all the oracle answers, and as such, form again distributions over $\cQ^{n/2}_{k+1,n}$.
\end{proof}
\begin{proof}[\thmref{thm:rlm_lb}]

Let $\cA$ be a oblivious stochastic CLI. By \lemref{lem:pcli_coefficients_polynomials_rlm} the coordinates of $\balpha^{(k)}_1$ (the point returned by the algorithm at the $k$'th iteration) when applied on the class of problems (\ref{problem:rlm}) distributes according to some distribution $\cD$ over $(\cQ^{n/2}_k)^n$. Furthermore, it is easy to verify that the corresponding minimizers of (\ref{problem:rlm}) are
\begin{align}\label{eq:rlm_minimizers}
	\balpha^*(\bpsi)&
	=\circpar{\frac{1}{ \frac{\lambda n +1}{\lambda n} + \frac{1}{\lambda n}\sin(\psi_1)}, \frac{1}{ \frac{\lambda n +1}{\lambda n} + \frac{1}{\lambda n}\sin(\psi_1)}   ,
	\frac{1}{ \frac{\lambda n +1}{\lambda n} + \frac{1}{\lambda n}\sin(\psi_2)}, \frac{1}{ \frac{\lambda n +1}{\lambda n} + \frac{1}{\lambda n}\sin(\psi_2)}   ,
	\dots
	}.
\end{align}
$\balpha^{(k)}_1(\bpsi)$ distributes according to some distribution $\cD$ over $\cQ^{n/2}_{k,n}$. Thus, by Yao principle, since each polynomial in $\cQ^{n/2}_{k,n}$ represents a single deterministic algorithm, we have 
\begin{align} \label{ineq:proof:thm:rlm_lb_ineq1}
	\max_{\bpsi\in\cH}\bE_{\balpha_1^{(k)}(\bpsi)\sim\cD}\|\balpha_1^{(k)}(\bpsi)-\balpha^*(\bpsi)\|\ge \min_{\bt(\bpsi)\in\cQ^{n/2}_{k,n}} \bE_{\bpsi\sim\cU(\cH)} \|\bt(	\bpsi)-\balpha^*(\bpsi)\|
\end{align}
where $\cU(\cH)$ denotes a distribution over $\cH$ which corresponds to of first drawing $j\sim\cU([n/2])$ at random, and then drawing $\psi_j$ according to distribution defined by the p.d.f. $p_{\psi_j}(\psi)=\cos(\psi)/2$ over $[-\pi/2,\pi/2]$ (for $i\neq j$ we set $\psi_i=0$ 
). We now have,
\begin{align}\label{ineq:proof:thm:rlm_lb_ineq2}
\min_{\bt(\bpsi)\in\cQ^{n/2}_{k,n}} &\bE_{\bpsi\sim\cU(\cH)} \|\bt(\bpsi)-\balpha^*(\bpsi)\|\nonumber\\
 &=
\min_{\bt(\bpsi)\in\cQ^{n/2}_{k,n}} \bE_{j\sim\cU([n/2])} \bE_{\bpsi_j\sim\cU([-\pi/2,\pi/2])} \|\bt(\psi)-\balpha^*(\bpsi)\| \nonumber\\
&=
\frac{2}{n}\sum_{j=1}^{n/2} \min_{\bt(\bpsi)\in\cQ^{n/2}_{k,n}}\bE_{\psi_j\sim\cU([-\pi/2,\pi/2])} \|\bt(\psi)-\balpha^*(\bpsi)\| \nonumber\\
&\ge
\frac{2}{n}\sum_{j=1}^{n/2} \min_{\bt(\bpsi)\in\cQ^{n/2}_{k,n}} \bE_{\psi_j\sim\cU([-\pi/2,\pi/2])}  \left|t_j(\psi_j)-\frac{1}{ \frac{\lambda n +1}{\lambda n} + \frac{1}{\lambda n}\sin(\psi_j)} \right|\nonumber\\
&\ge
\frac{1}{n}\sum_{j=1}^{n/2} \min_{\bt(\bpsi)\in\cQ^{n/2}_{k,n}}  \int_{-\pi/2}^{\pi/2} \left|t_j(\psi_j)-\frac{1}{ \frac{\lambda n +1}{\lambda n} + \frac{1}{\lambda n}\sin(\psi_j)} \right|\cos\psi_j ~d\psi_j\nonumber\\
&=
\frac{1}{n}\sum_{j=1}^{n/2} \min_{\bs(\bpsi)\in\cQ^{n/2}_{k,n}} \int_{-1}^{1} \left|s_j(\eta_j)-\frac{1}{ \frac{\lambda n +1}{\lambda n} + \frac{1}{\lambda n}\eta_j} \right| ~d\eta_j\nonumber\\
&=
\lambda\sum_{j=1}^{n/2} \min_{\bs(\bpsi)\in\cQ^{n/2}_{k,n}} \int_{-1}^{1} \left|s_j(\eta_j)-\frac{1}{ \lambda n +1 + \eta_j} \right| ~d\eta_j
\end{align}
where the first inequality follows by focusing on the $j$'th 
coordinate of $\bs(\psi)-\balpha^*(\bpsi)$ in each summand. Now, set $\alpha=1+\lambda n, L= 3,\mu=1$  and note that
\begin{align*}
\sqrt\frac{2\alpha+L-\mu}{2\alpha+\mu-L}=
\sqrt\frac{2\lambda n +4}{2\lambda n}=
\sqrt\frac{\lambda n +2}{\lambda n}=
\sqrt{\frac{2}{\lambda n}+1}
\end{align*}
Thus, by \lemref{appendix:lem:l1_norm_lb}, using the same value for $\alpha$ and noting that $\alpha>1=(L-\mu)/2$) yields 
\begin{align*}
	\int_{-1}^{1} \left|s_j(\eta_j)-\frac{1}{ \lambda n +1 + \eta_j} \right| ~d\eta_j\ge 
	\circpar{\frac{\sqrt{\frac{2}{\lambda n}+1}-1}{\sqrt{\frac{2}{\lambda n}+1}+1}}^{k_j}
\end{align*}
where $k_j$ denotes the degree of $s_j(\eta_j)$. Plugging in this into \ineqref{ineq:proof:thm:rlm_lb_ineq2} we get 
\begin{align*}
\max_{\bpsi\in\cH}\bE_{\balpha_1^{(k)}(\bpsi)\sim\cD}\|\balpha_1^{(k)}(\bpsi)-\balpha^*(\bpsi)\|&\ge
\lambda\min_{\bs(\bpsi)\in\cQ^{n/2}_{k,n}}\sum_{j=1}^{n/2} \circpar{\frac{\sqrt{\frac{2}{\lambda n}+1}-1}{\sqrt{\frac{2}{\lambda n}+1}+1}}^{k_j}.
\end{align*}
Since $u\mapsto\rho^u$ is a decreasing and convex function for any $1>\rho>0$, we have
\begin{align*}
\lambda\min_{\bs(\bpsi)\in\cQ^{n/2}_{k,n}}\sum_{j=1}^{n/2} \circpar{\frac{\sqrt{\frac{2}{\lambda n}+1}-1}{\sqrt{\frac{2}{\lambda n}+1}+1}}^{k_j}&\ge
n\lambda/2\min_{\bs(\bpsi)\in\cQ^{n/2}_{k,n}} \circpar{\frac{\sqrt{\frac{2}{\lambda n}+1}-1}{\sqrt{\frac{2}{\lambda n}+1}+1}}^{\frac{2}{n}\sum_{j=1}^{n/2}k_j}
\\&\ge
n\lambda/2\min_{\bs(\bpsi)\in\cQ^{n/2}_{k,n}} \circpar{\frac{\sqrt{\frac{2}{\lambda n}+1}-1}{\sqrt{\frac{2}{\lambda n}+1}+1}}^{\frac{2k}{n}}
\end{align*}
where the last inequality is due to the fact that $\bs(\bpsi)\in\cQ^{n/2}_{k,n}(\sin\psi)$ which implies that $\sum_{j=1}^n k_j \le k$.
Finally, we have,
\begin{align*}
	\max_{\bpsi\in\cH}\bE_{\balpha_1^{(k)}(\bpsi)\sim\cD}[D_{\bpsi}(\balpha_1^{(k)}(\bpsi))-D_{\bpsi}(\balpha^*(\bpsi))]&\ge
	\max_{\bpsi\in\cH}\bE_{\balpha_1^{(k)}(\bpsi)\sim\cD}\frac{1}{2}\|\balpha_1^{(k)}(\bpsi)-\balpha^*(\bpsi)\|^2\\
	&\ge\frac{1}{2} \circpar{\max_{\bpsi\in\cH}\bE_{\balpha_1^{(k)}(\bpsi)\sim\cD}\|\balpha_1^{(k)}(\bpsi)-\balpha^*(\bpsi)\|}^2\\
	&\ge 
\frac12\circpar{n\lambda/2\min_{\bs(\bpsi)\in\cQ^{n/2}_{k,n}} \circpar{\frac{\sqrt{\frac{2}{\lambda n}+1}-1}{\sqrt{\frac{2}{\lambda n}+1}+1}}^{\frac{2k}{n}}}^2,
\end{align*}
where the first inequality follows by the $1$-strong convexity of $D_{\bpsi}$ and the third inequality follows by Jensen inequality. Using \lemref{lem:technical2tag}, we get that the iteration complexity of $\cA$ is at least 
\begin{align*}
	\frac{1}{8}\sqrt\frac{2n}{\lambda } \circpar{\ln\frac{n^2\lambda^2}{8} + \ln(1/\epsilon ) }.
\end{align*}

\end{proof}

Lastly, we bound from below the number of iterations required to obtain a non-trivial accuracy. 
\begin{lemma}\label{lemma:n_iteration_rlm}
Let $j\in[n]$, let $\bpsi_{j,1}\in\cH$ 
be a parameters vector whose all coordinates are $-\pi/2$ and let $\bbeta_{\bpsi,2}\in\cH$ be a parameters vector whose all coordinates are $-\pi/2$, except for the $j$'th coordinate which we set to be $\pi/2$. Then  
\begin{align*}
	\|\balpha^*(\bpsi_1) -  \balpha^*(\bpsi_2)\|\ge \frac{2\sqrt2 }{ \lambda n+2}
\end{align*} 
\end{lemma}
\begin{proof}
By \eqref{eq:rlm_minimizers} we have
\begin{align*}
	\|\balpha^*(\bpsi_1) -  \balpha^*(\bpsi_2)\|&= 
	 \sqrt2\circpar{\frac{1}{\frac{\lambda n+1}{\lambda n} -\frac{1}{\lambda n} }- \frac{1}{\frac{\lambda n+1}{\lambda n} +\frac{1}{\lambda n} }}  \\
	&= \sqrt2\circpar{1- \frac{\lambda n }{ \lambda n+2}}\\
	&= \frac{2\sqrt2 }{ \lambda n+2}.
\end{align*} 

\end{proof}

\begin{theorem}
When applied on (\ref{problem:rlm}) ,the iteration complexity of oblivious stochastic CLI algorithms equipped with a dual RLM oracle $D_{\bpsi}$ is at least $n/2$.
\end{theorem}
\begin{proof}
Let $\cA$ be a stochastic optimization algorithm. By \lemref{lem:pcli_coefficients_polynomials_rlm} the coordinates of $\balpha^{(k)}_1$ (the point returned by the algorithm at the $k$'th iteration) when applied on the class of problems (\ref{problem:rlm}) distributes according to some distribution $\cD$ over $(\cQ^{n/2}_k)^n$. By Yao principle, since each polynomial in $\cQ^{n/2}_{k,n}$ represents a single deterministic algorithm, we have 
\begin{align} \label{ineq:proof:thm:rlm_lb_n_iterations}
	\max_{\bpsi\in\cH}\bE_{\balpha_1^{(k)}(\bpsi)\sim\cD}\|\balpha_1^{(k)}(\bpsi)-\balpha^*(\bpsi)\|\ge \min_{\bt(\bpsi)\in\cQ^{n/2}_{k,n}} \bE_{\bpsi\sim\cD(\cH)} \|\bt(	\bpsi)-\balpha^*(\bpsi)\|
\end{align}
where $\cD(\cH)$ denotes a distribution over $\cH$ which corresponds to the process of first drawing $j\sim\cU([n/2])$ at random, and then set $\bpsi$ to be $\bpsi_{j,1}$ or $\bpsi_{j,2}$ as defined in \lemref{lemma:n_iteration_rlm} with equal probability. Now, for $k<n/2$, there exists some $j\in[n/2]$ such that $\bt(\bpsi)$ does not depend on $\psi_j$.
This yields,  
\begin{align*}
	\max_{\{\bpsi_{j,i}|j\in[n/2],i\in[2]\}} &\bE_{\cA}[ D_{\bpsi_{j,i}}(\balpha^{(k)}\circpar{\bpsi_{j,i}}) - D_{\bpsi_{j,i}}(\balpha^{*}\circpar{\bpsi_{j,i}}) ] \\&\ge 
	\min_{\text{deterministic algorithms}}\bE_{_{\bpsi}\sim\cD(\cH)}[ D_{\bpsi_{j,i}}(\balpha^{(k)}\circpar{\bpsi_{j,i}}) - D_{\bpsi_{j,i}}(\balpha^{*}\circpar{\bpsi_{j,i}}) ]\\
	&\ge 
	\min_{\text{deterministic algorithms}}\bE_{_{\bpsi}\sim\cD(\cH)} 
	\frac{1}{2}\| \balpha^{(k)}\circpar{\bpsi_{j,i}}) - \balpha^{*}\circpar{\bpsi_{j,i}} \|^2\\
	&\ge 
	\frac12\min_{\text{deterministic algorithms}}\circpar{\bE_{_{\bpsi}\sim\cD(\cH)} 
	\| \balpha^{(k)}\circpar{\bpsi_{j,i}}) - \balpha^{*}\circpar{\bpsi_{j,i}} \|}^2\\
&\ge \frac12\circpar{\frac{2\sqrt2 }{ n(\lambda n+2)}}^2,
\end{align*}
where the last inequality follows from \lemref{lemma:n_iteration_rlm}. Thus, for sufficiently small $\epsilon$, one must perform at least $n/2$ iterations in order to obtain an $\epsilon$-optimal solution.
\end{proof}

						\subsection{Best polynomial approximation over closed intervals in \texorpdfstring{$\reals$}{reals} } 
In the following section we analyze the best polynomial approximation of some functions w.r.t. $L_\infty$, $L_1$ and $L_2$ norm, based on standard results from the approximation theory (see generally,
Allan Pinkus. On L1-approximation, 1989; Theodore J Rivlin. An introduction to the approximation of functions, 2003; Ronald A DeVore and George G Lorentz. Constructive approximation, 1993; Naum Il’ich Akhiezer and Charles J Hyman. Theory of approximation. Translated by Charles J. Hyman. New York, 1956; Isidor Pavlovich Natanson. Constructive function theory, 1964).

										\subsubsection{Approximation w.r.t. \texorpdfstring{$L_\infty$}{maximum norm} } \label{subsubsection:max_norms}
										
\begin{lemma} \label{lem:max_norm_general}
Let $b>a>0$ and $c>-a$, then 
\begin{align*}
	\min_{s(\eta)\in\cP_k}\left\|s(\eta)-\frac{1}{\eta+c} \right\|_{L_{\infty}([a,b])} \ge\frac{2(b-a)}{(b+a+2c)^2-(b-a)^2}\circpar{\frac{\sqrt\frac{b+c}{a+c}-1}{\sqrt\frac{b+c}{a+c}+1}}^k.
\end{align*}
\end{lemma}
\begin{proof}
We have,
\begin{align*}
	\min_{s(\eta)\in\cP_k}\left\|s(\eta)-\frac{1}{\eta+c} \right\|_{L_{\infty}([a,b])} 
	&=\min_{s(\eta)\in\cP_k}\left\|s\circpar{\frac{a-b}{2}\eta+\frac{a+b}{2}}-\frac{1}{\frac{a-b}{2}\eta+\frac{b+a}{2}+c} \right\|_{L_{\infty}([-1,1])} \\
	&=\min_{s(\eta)\in\cP_k}\left\|s\circpar{\eta}-\frac{1}{\frac{a-b}{2}\eta+\frac{b+a+2c}{2}} \right\|_{L_{\infty}([-1,1])} \\
	&=\frac{2}{b-a}\min_{s(\eta)\in\cP_k}\left\|\frac{b-a}{2} s\circpar{\eta}-\frac{\frac{b-a}{2}}{\frac{a-b}{2}\eta+\frac{b+a+2c}{2}} \right\|_{L_{\infty}([-1,1])} \\
	&=\frac{2}{b-a}\min_{s(\eta)\in\cP_k}\left\|s\circpar{\eta}+\frac{1}{\eta-\frac{b+a+2c}{b-a}} \right\|_{L_{\infty}([-1,1])} \\
	&=\frac{2}{b-a}\min_{s(\eta)\in\cP_k}\left\|-s\circpar{\eta}+\frac{1}{\eta-\frac{b+a+2c}{b-a}} \right\|_{L_{\infty}([-1,1])} \\
	&=\frac{2}{b-a}\min_{s(\eta)\in\cP_k}\left\|s\circpar{\eta}-\frac{1}{\eta-\frac{b+a+2c}{b-a}} \right\|_{L_{\infty}([-1,1])} 
\end{align*}
where  we used the fact that $\cP_k$ is invariant under pre-composition and post-composition with linear function in the second, fourth and fifth equalities. Now, since 
\begin{align*}
	c>-a\implies \frac{b+a+2c}{b-a} >1,
\end{align*}
combining Inequality \ref{ineq:basic_chebyshev} with \lemref{lem:technical1}, yields
\begin{align} 
	 \min_{s(\eta)\in\cP_k}  \left\| s(\eta)- \frac{1}{\eta-\frac{b+a+2c}{b-a} }\right\|_{L_{\infty}([-1,1])}	&\ge
	\frac{\circpar{\circpar{\frac{b+a+2c}{b-a}}-\sqrt{\circpar{\frac{b+a+2c}{b-a}}^2-1}}^k}{\circpar{\frac{b+a+2c}{b-a}}^2-1}\nonumber\\&=
	\frac{1}{\circpar{\frac{b+a+2c}{b-a}}^2-1}\circpar{\frac{1-\sqrt\frac{\frac{b+a+2c}{b-a}-1}{\frac{b+a+2c}{b-a}+1}}{1+\sqrt\frac{\frac{b+a+2c}{b-a}-1}{\frac{b+a+2c}{b-a}+1}}}^k
\end{align}
Noting that 
\begin{align*}
\frac{\frac{b+a+2c}{b-a}  -1}{\frac{b+a+2c}{b-a} +1}=\frac{b+a+2c - (b-a)}{b+a+2c+b-a}=\frac{a+c}{b+c},
\end{align*}
we get 
\begin{align*}
	\min_{s(\eta)\in\cP_k}\left\|s(\eta)-\frac{1}{\eta+c} \right\|_{L_{\infty}([a,b])} &\ge \frac{2}{b-a}\frac{1}{\circpar{\frac{b+a+2c}{b-a}}^2-1}\circpar{\frac{\sqrt\frac{b+c}{a+c}-1}{\sqrt\frac{b+c}{a+c}+1}}^k.
\end{align*}
As,
\begin{align*}
\frac{2}{b-a}\frac{1}{\circpar{\frac{b+a+2c}{b-a}}^2-1}=\frac{2}{b-a}\frac{(b-a)^2}{(b+a+2c)^2-(b-a)^2}=
\frac{2(b-a)}{(b+a+2c)^2-(b-a)^2}
\end{align*}
we get 
\begin{align*}
	\min_{s(\eta)\in\cP_k}\left\|s(\eta)-\frac{1}{\eta+c} \right\|_{L_{\infty}([a,b])}&\ge
	\frac{2(b-a)}{(b+a+2c)^2-(b-a)^2}\circpar{\frac{\sqrt\frac{b+c}{a+c}-1}{\sqrt\frac{b+c}{a+c}+1}}^k.
\end{align*}

\end{proof}

											\subsubsection{Approximation w.r.t. \texorpdfstring{$L_1$}{L one norm} } \label{subsubsection:L1_norms}
Let $U_k(\eta)$ denote the $k$'th second order Chebyshev polynomial, i.e.,
\begin{align}\label{def:chebyshev_second}
	U_k(\eta)&\coloneqq \frac{\sin((k+1)\arccos{\eta})}{\sqrt{1-\eta^2}}
\end{align}
(To see why these are indeed polynomials, observe that $U_k(\eta)$ are the derivative of Chebyshev polynomials of first order scaled by a factor of $1/k$). The zeros of $U_k(\eta)$ are $\eta_j=\cos(\frac{j\pi}{k+1}),~j=1,\dots,k$. First, let us establish the orthogonality of $\sgn(U_k(\eta))$ with respect to $\cP_{k-1}$ over $[-1,1]$.

\begin{lemma}
Let $p(\eta)\in\cP_{k-1}$, then 
\begin{align*}
	\int_{-1}^1 p(\eta)\sgn(U_k(\eta))~d\eta=0
\end{align*}
\end{lemma}
\begin{proof}
We integrate by substituting $\eta=\frac{e^{i\theta}+e^{-i\theta}}{2}\circpar{=\cos(\theta)}$,  
\begin{align} \label{eq:lem4_ortho}
	\int_{-1}^1 p(\eta)\sgn(U_k(\eta))~d\eta~&\stackrel{\eta=\frac{e^{i\theta}+e^{-i\theta}}{2}}{=} \int_{\pi}^0 p\circpar{\frac{e^{i\theta}+e^{-i\theta}}{2}}\sgn\circpar{U_k\circpar{\frac{e^{i\theta}+e^{-i\theta}}{2}}}\circpar{\frac{ie^{i\theta}-ie^{-i\theta}}{2}} ~d\theta\nonumber\\
	&=\int_0^{\pi} p\circpar{\frac{e^{i\theta}+e^{-i\theta}}{2}}\sgn\circpar{ \frac{\sin((k+1)\theta)}{\sin(\theta)} }\circpar{\frac{ie^{-i\theta}-ie^{i\theta}}{2}} ~d\theta\nonumber\\
	&=\int_0^{\pi} p\circpar{\frac{e^{i\theta}+e^{-i\theta}}{2}}\sgn\circpar{ \sin((k+1)\theta)}\circpar{\frac{ie^{-i\theta}-ie^{i\theta}}{2}} ~d\theta\nonumber\\
	&=\frac{1}{2}\int_{-\pi}^\pi p\circpar{\frac{e^{i\theta}+e^{-i\theta}}{2}}\sgn\circpar{ \sin((k+1)\theta)}\circpar{\frac{ie^{-i\theta}-ie^{i\theta}}{2}} ~d\theta
\end{align}
where the last equality is due to the fact that the integrand is an even function in $\theta$. Lastly, since for any $j=1,\dots,k$ we have
\begin{align*}
	\int_{-\pi}^\pi e^{-ij\theta} \sin((k+1)\theta)  d\theta&=\int_{-\pi+2\pi/(k+1)}^{\pi+2\pi/(k+1)} e^{-ij\theta} \sin((k+1)\theta)  d\theta\\
	&=\int_{-\pi}^{\pi} e^{-ij(\theta+2\pi/(k+1))} \sin((k+1)(\theta+2\pi/(k+1)))  d\theta\\
	&=e^{-2\pi i j/(k+1)} \int_{-\pi}^{\pi} e^{-ij\theta} \sin((k+1)\theta)  d\theta,
\end{align*}
and since $e^{-2\pi i j/(k+1)}\neq1$ for $j=1,\dots,k$, it follows that 
\begin{align*}
\int_{-\pi}^\pi e^{-ij\theta} \sin((k+1)\theta)  d\theta=0, \quad j=1,\dots,k.
\end{align*}
This, together with the case where $j=0$,
\begin{align*}
	\int_{-\pi}^\pi \sin((k+1)\theta)  d\theta = \circpar{-\cos((k+1)\theta)/(k+1)}\Big|_{-\pi}^{\pi}=0,
\end{align*}
implies that all the terms in (\ref{eq:lem4_ortho}) vanish, thus concluding the proof.
\end{proof}

Given $\mu<L$ (note that, here $\mu$ and $L$ are allowed to take negative values), we define 
\begin{align*}
	\tilde{U}_k(\eta)\coloneqq U_k\circpar{ \frac{2\eta}{\mu-L}}.
\end{align*}
By substituting $\eta$ for $\frac{(\mu-L)\eta}{2}$, we get the following corollary.
\begin{corollary} \label{appendix:cor_ortho}
Let $p(\eta)\in\cP_{k-1}$, then
\begin{align} \label{eq:L_1_ortho_mu_L}
	\int_{-\frac{L-\mu}{2}}^\frac{L-\mu}{2} p(\eta)~\sgn(\tilde{U}_k(\eta))~d\eta =0
\end{align}
\end{corollary}
We now use \corref{appendix:cor_ortho} to bound from below the best polynomial $L_1$-approximation error w.r.t. $1/(\eta+\alpha)$ over the interval $[\mu,L]$.

\begin{lemma} \label{appendix:lem:l1_norm_lb}
Let $p(\eta)\in\cP_{k-1}$. Then, for any $(L-\mu)/2<\alpha$ we have
\begin{align*}
	\int_{-\frac{L-\mu}{2}}^{\frac{L-\mu}{2}} |p(\eta)-1/(\eta+\alpha)|d\eta &\ge\circpar{\frac{\sqrt\frac{2\alpha+L-\mu}{2\alpha+\mu-L}-1}{\sqrt\frac{2\alpha+L-\mu}{2\alpha+\mu-L}+1}}^k.
\end{align*}
\end{lemma}
\begin{proof}
First, note that the following two inequalities 
\begin{align*}
	\int_{-\frac{L-\mu}{2}}^{\frac{L-\mu}{2}} |p(\eta)-1/(\eta+\alpha)|d\eta 
		&\ge\int_{-\frac{L-\mu}{2}}^{\frac{L-\mu}{2}} (p(\eta)-1/(\eta+\alpha))\sgn(\tilde{U}_k(\eta)) d\eta=-\int_{-\frac{L-\mu}{2}}^{\frac{L-\mu}{2}} 1/(\eta+\alpha)\sgn(\tilde{U}_k(\eta)) d\eta\\
		\int_{-\frac{L-\mu}{2}}^{\frac{L-\mu}{2}} |p(\eta)-1/(\eta+\alpha)|d\eta 
		&\ge\int_{-\frac{L-\mu}{2}}^{\frac{L-\mu}{2}} (p(\eta)-1/(\eta+\alpha))\sgn(-\tilde{U}_k(\eta)) d\eta =\int_{-\frac{L-\mu}{2}}^{\frac{L-\mu}{2}} 1/(\eta+\alpha)\sgn(\tilde{U}_k(\eta)) d\eta\\
\end{align*}
hold due to orthogonality condition (\ref{eq:L_1_ortho_mu_L}) and the fact that $\sgn(\cdot)$ is odd. Therefore, 
\begin{align*}
	\int_{-\frac{L-\mu}{2}}^{\frac{L-\mu}{2}} |p(\eta)-1/(\eta+\alpha)|d\eta 
		&\ge\left|\int_{-\frac{L-\mu}{2}}^{\frac{L-\mu}{2}} 1/(\eta+\alpha)~\sgn(\tilde{U}_k(\eta)) d\eta\right|.
\end{align*}
Substituting $\frac{\mu-L}{2}\eta$ for $\eta$, yields
\begin{align}\label{eq:L1_approx_aux1323}
	\left|\int_{-\frac{L-\mu}{2}}^{\frac{L-\mu}{2}} 1/(\eta+\alpha)~\sgn(\tilde{U}_k(\eta)) d\eta\right|&=\left|\int_{1}^{-1} \frac{1}{\frac{\mu-L}{2}\eta + \alpha}~\sgn(U_k(\eta)) \frac{\mu-L}{2} d\eta\right|\nonumber\\
	&=\left|\int_{-1}^1 \frac{1}{\frac{\mu-L}{2}\eta+ \alpha}~\sgn(U_k(\eta)) \frac{L-\mu}{2} d\eta\right|\nonumber\\
	&=\left|\int_{-1}^1 \frac{1}{-\eta+\frac{2\alpha}{L-\mu} }~\sgn(U_k(\eta)) d\eta\right|
\end{align}

Now, plugging in the definition of $U_k(\eta)$ (see  (\ref{def:chebyshev_second})) and applying  \lemref{lem:technical3}, we get
\begin{align*}
	\int_{-1}^1\frac{\sgn{(\sin(k\arccos(\eta)))}}{u-\eta}d\eta&\ge \circpar{\frac{1-\sqrt\frac{u-1}{u+1}}{1+\sqrt\frac{u-1}{u+1}}}^k
\end{align*}
for any $u>1$. Using this inequality with (\ref{eq:L1_approx_aux1323}) where $u= \frac{2\alpha}{L-\mu}$, yields
\begin{align*}
	\int_{-\frac{L-\mu}{2}}^\frac{L-\mu}2 |p(\eta)-1/(\eta+\alpha)|d\eta&\ge
	\circpar{\frac{\sqrt\frac{2\alpha+L-\mu}{2\alpha+\mu-L}-1}{\sqrt\frac{2\alpha+L-\mu}{2\alpha+\mu-L}+1}}^k.
\end{align*}
\end{proof}

								\subsubsection{Approximation w.r.t. \texorpdfstring{$L_2$}{L two norm}} \label{subsubsection:L2_norms}

\begin{lemma}\label{appendix:lem:l2_norm_lb}
For any $\alpha\in(-1,0)$,
	\begin{align*}
	\min_{s(\eta)\in\cP_{k-1}} \int_{0}^1 \eta(s(\eta)\eta-1)^2 \eta^\alpha d\eta \ge 
	\frac{1}{e^{2} (k+2)^{2(\alpha+1)+2}}
\end{align*}
\end{lemma}								

\begin{proof}								
Rephrasing it equivalently as 
\begin{align*}
	\min_{s(\eta)\in\cP_{k-1}} \int_{0}^1 (s(\eta)\eta^{\frac{3+\alpha}{2}}-\eta^{\frac{1+\alpha}{2}})^2 d\eta,
\end{align*}
shows that this problem can be seen as a best $L_2$-approximation for $\eta^{\frac{1+\alpha}{2}}$ in the $k$-dimensional space spanned by $g_i = \eta^{i+\frac{1+\alpha}{2}},~i=1,\dots,k$ (accordingly, $g_0=\eta^{\frac{1+\alpha}{2}}$). By \cite[Equation (3), p. 16]{akhiezer1956theory}, we have
\begin{align*}
	\min_{s(\eta)\in\cP_{k-1}} \int_{0}^1 (s(\eta)\eta^{\frac{3+\alpha}{2}}-\eta^{\frac{1+\alpha}{2}})^2 d\eta,
 = \frac{\det G(g_0,g_1,\dots,g_n)}{\det G(g_1,\dots,g_n)}
\end{align*}
where $G(\cdot)$ is Gram matrix (whose entries are the inner products of its arguments). First, note that 
\begin{align*}
\inner{g_i,g_j}=\int_{0}^1 \eta^{i+\frac{1+\alpha}{2}}\eta^{j+\frac{1+\alpha}{2}} d\eta=\int_{0}^1 \eta^{i+j+1+\alpha}d\eta = \frac{1}{i+j+\alpha+2},~i,j=0,1,\dots,k
\end{align*}
Thus,
\begin{align*}
	G(g_1,\dots,g_k)_{i,j} &= \frac{1}{i+j+\alpha+2},\\
	G(g_0,\dots,g_k)_{i,j} &=  \frac{1}{i+j+\alpha}.
\end{align*}
It follows that both matrices can be expressed as a Cauchy matrices, that is 
\begin{align*}
	G(g_1,\dots,g_k)_{i,j} &= \frac{1}{x_i-y_j},\\
	G(g_0,\dots,g_k)_{i,j} &=  \frac{1}{u_i-v_j}.
\end{align*}
where $x_i=i+\alpha+1,y_j=-j-1,~i,j\in[k]$ and $u_i=i+\alpha,v_j=-j,~i,j\in[k+1]$. The determinant of Cauchy matrix $A$ defined by sequences $w_i,z_j$ one has
\begin{align*}
	\det A = \frac{\prod_{i=2}^{k}\prod_{j=1}^{i-1} (w_i-w_j)(z_j-z_i)}{\prod_{i=1}^k\prod_{j=1}^k(w_i-z_j)}
\end{align*}
Hence, 
\begin{align*}
	\frac{\det G(g_0,g_1,\dots,g_k)}{\det G(g_1,\dots,g_k)}&=\frac{   \frac{\prod_{i=2}^{k+1}\prod_{j=1}^{i-1} (u_i-u_j)(v_j-v_i)}{\prod_{i=1}^{k+1}\prod_{j=1}^{k+1}(u_i-v_j)} }{\frac{\prod_{i=2}^{k}\prod_{j=1}^{i-1} (x_i-x_j)(y_j-y_i)}{\prod_{i=1}^k\prod_{j=1}^k(x_i-y_j)}}\\
	&=\frac{   \frac{\prod_{i=2}^{k+1}\prod_{j=1}^{i-1} (i-j)(-j-(-i))}{\prod_{i=1}^{k+1}\prod_{j=1}^{k+1}(i+\alpha-(-j))} }{\frac{\prod_{i=2}^{k}\prod_{j=1}^{i-1} ((i+1)-(j+1))((-j-1)-(-i-1))}{\prod_{i=1}^k\prod_{j=1}^k((i+\alpha+1)-(-j-1))}}\\
	&=\frac{   \frac{\prod_{i=2}^{k+1}\prod_{j=1}^{i-1} (i-j)^2}{\prod_{i=1}^{k+1}\prod_{j=1}^{k+1}(i+j+\alpha)} }{\frac{\prod_{i=2}^{k}\prod_{j=1}^{i-1} (i-j)^2}{\prod_{i=1}^k\prod_{j=1}^k(i+j+\alpha+2)}}\\
	&=\frac{\prod_{i=2}^{k+1}\prod_{j=1}^{i-1} (i-j)^2}{\prod_{i=1}^{k+1}\prod_{j=1}^{k+1}(i+j+\alpha)} 
	\frac{\prod_{i=1}^k\prod_{j=1}^k(i+j+\alpha+2)}{\prod_{i=2}^{k}\prod_{j=1}^{i-1} (i-j)^2}\\
	&=\frac{\prod_{j=1}^{k} (k+1-j)^2 \prod_{i=1}^k\prod_{j=1}^k(i+j+\alpha+2)}{\prod_{i=1}^{k+1}\prod_{j=1}^{k+1}(i+j+\alpha)}  \\
	&=\frac{\prod_{j=1}^{k} (k+1-j)^2 \prod_{i=1}^k\prod_{j=1}^k((i+\alpha+1)+(j+1))}{\prod_{i=1}^{k+1}\prod_{j=1}^{k+1}((i+\alpha)+j)}  \\
	&=\frac{\prod_{j=1}^{k} (k+1-j)^2 \prod_{i=2}^{k+1}\prod_{j=2}^{k+1}((i+\alpha)+j)}{\prod_{i=1}^{k+1}\prod_{j=1}^{k+1}((i+\alpha)+j)}  \\
	&=\frac{\prod_{j=1}^{k} (k+1-j)^2 }{  \prod_{i=1}^{k+1}(i+\alpha+1)\prod_{j=2}^{k+1}(1+\alpha+j)}  \\ 
	&=\frac{\prod_{j=1}^{k} j^2 }{  \prod_{i=1}^{k+1}(i+\alpha+1)\prod_{j=2}^{k+1}(1+\alpha+j)}  \\
	&=\frac{(\alpha+2)\prod_{j=1}^{k} j^2 }{  \prod_{i=1}^{k+1}(i+\alpha+1)^2}  \\
	&=(\alpha+2)\circpar{\prod_{j=1}^k\frac{j}{j+\alpha+1}}^2\frac{1}{(k+\alpha+2)^2}
\end{align*}
To estimate the middle term, we apply arguments similar to the integral test for infinite series. First, note that,
\begin{align*}
	\prod_{j=1}^k\frac{j}{j+\alpha+1}=\exp\circpar{\sum_{i=1}^k \ln \frac{j}{j+\alpha+1}}.
\end{align*}
Now, since for any $\alpha\in(-1,0)$, it holds that $x\mapsto\ln\frac{x}{x+\alpha+1}$ is a monotone decreasing function (over $x\neq\alpha$), 
it holds that 
\begin{align*}
	\sum_{i=1}^k \ln \frac{j}{j+\alpha+1}&\ge\int_{1}^{k+1} \ln\frac{x}{x+\alpha+1}dx=\left( x\ln\frac{x}{\alpha + x+1} - (\alpha+1)\ln(\alpha + x+1) \middle)\right|_{1}^{k+1} 
\end{align*}
Hence,
\begin{align*}
	\prod_{j=1}^k\frac{j}{j+\alpha+1}&\ge 
	\exp\circpar{ (k+1)\ln\frac{k+1}{\alpha + (k+1)+1} - (\alpha+1)\ln(\alpha + (k+1)+1)
	- \ln\frac{1}{\alpha + 2} + (\alpha+1)\ln(\alpha + 2) }\\
	&=\circpar{\frac{k+1}{k+\alpha+2}}^{k+1}(k+\alpha+2)^{-(\alpha+1)}(\alpha+2)^{\alpha+2}\\
	&\ge\circpar{\frac{k+1}{k+\alpha+2}}^{k+1}(k+\alpha+2)^{-(\alpha+1)}\\
	&=\circpar{1 - \frac{\alpha+1}{k+\alpha+2}}^{k+1}(k+\alpha+2)^{-(\alpha+1)}\\
	&=\circpar{1 - \frac{1}{\frac{k+1}{\alpha+1}+1}}^{k+1}(k+\alpha+2)^{-(\alpha+1)}.
\end{align*}
Now, by the following standard inequality
\begin{align*}
	1-\frac{2}{x+1}&\ge \exp\circpar{\frac{-2}{x-1} },
\end{align*}
we get, 
\begin{align*}
	1 - \frac{1}{\frac{k+1}{\alpha+1}+1}&=1 - \frac{2}{(2\frac{k+1}{\alpha+1}+1)+1}\ge\exp\circpar{\frac{-2}{(2\frac{k+1}{\alpha+1}+1)-1} }
	=\exp\circpar{\frac{-1}{\frac{k+1}{\alpha+1}} }=\exp\circpar{\frac{-(\alpha+1)}{k+1} }
\end{align*}
therefore,
\begin{align*}
	(\alpha+2)\circpar{\prod_{j=1}^k\frac{j}{j+\alpha+1}}^2\frac{1}{(k+\alpha+2)^2}
&\ge(\alpha+2)\circpar{\exp\circpar{\frac{-(\alpha+1)}{k+1} }^{k+1}(k+\alpha+2)^{-(\alpha+1)}}^2\frac{1}{(k+\alpha+2)^2}
\\&=(\alpha+2)\exp\circpar{-2(\alpha+1) }(k+\alpha+2)^{-2(\alpha+1)-2}
\\&\ge(\alpha+2)\exp\circpar{-2(\alpha+1) }(k+2)^{-2(\alpha+1)-2}
\end{align*}
All in all, we get
\begin{align*}
	\min_{s(\eta)\in\cP_{k-1}} \int_{0}^1 \eta(s(\eta)\eta-1)^2 \eta^\alpha d\eta&\ge
	\frac{1}{e^{2} (k+2)^{2(\alpha+1)+2}}
\end{align*}																																					

\end{proof}

																\subsection{Technical Lemmas}

\begin{lemma}\label{lem:technical1} 
For any $u\ge1$, 
\begin{align*}
	u-\sqrt{u^2-1}=\frac{1-\sqrt\frac{u-1}{u+1}}{1+\sqrt\frac{u-1}{u+1}}.
\end{align*}
\end{lemma}
\begin{proof}
We have,
\begin{align*}
	\frac{1-\sqrt\frac{u-1}{u+1}}{1+\sqrt\frac{u-1}{u+1}}&=	\frac{\circpar{1-\sqrt\frac{u-1}{u+1}}^2}{1-\frac{u-1}{u+1}}
	=	\frac{(u+1)\circpar{1-\sqrt\frac{u-1}{u+1}}^2}{u+1-(u-1)	}
	=	\frac{\circpar{\sqrt{u+1}-\sqrt{u-1}}^2}{2}\\
	&=	\frac{u+1 -2\sqrt{(u+1)(u-1)} +(u-1) }{2}
	=	u-\sqrt{u^2-1}
\end{align*}

\end{proof}

%
%

\begin{lemma} \label{lem:technical3}
For any $u>1$,
\begin{align*}
		\int_{-1}^1\frac{\sgn{(\sin(k\arccos(\eta)))}}{u-\eta}d\eta&\ge \circpar{\frac{1-\sqrt\frac{u-1}{u+1}}{1+\sqrt\frac{u-1}{u+1}}}^k.
\end{align*}
\end{lemma}
\begin{proof}
First, note that the function 
\begin{align*}
	\gamma(x)&\coloneqq \ln\frac{x+1}{x-1} -\frac1x	
\end{align*}
takes non-negative for any $x>1$, as
\begin{align*}
	\gamma'(x) &=  \frac{x-1}{x+1}\frac{ x-1 - (x+1) }{(x-1)^2} + \frac{1}{x^2}
	=  \frac{x-1}{x+1}\frac{ -2 }{(x-1)^2} + \frac{1}{x^2}\\
	&\le  \frac{ -2 }{(x-1)^2} + \frac{1}{x^2}
	\le	\frac{-1}{(x-1)^2}<0
\end{align*}
and $\lim_{x\to\infty} \gamma(x)=0$. Therefore, by using identity (see Section F.31. in \cite{akhiezer1956theory}), we get 
\begin{align*}
	\int_{-1}^1\frac{\sgn{(\sin(k\arccos(\eta)))}}{u-\eta}d\eta
	&=2\ln\frac{ (u+\sqrt{u^2-1})^k + 1}{ (u+\sqrt{u^2-1})^k -1}\\
	&\ge (u-\sqrt{u^2-1})^k=\circpar{\frac{1-\sqrt\frac{u-1}{u+1}}{1+\sqrt\frac{u-1}{u+1}}}^k,
\end{align*}
where the last equality is due to \lemref{lem:technical1}. 
\end{proof}

\begin{lemma}\label{lem:technical2tag} 
Let $L>\mu>0$, $c>0$ and $\alpha\ge0$. Then
\begin{align*}
	\epsilon&\ge	c\circpar{\frac{\sqrt\frac{L+\alpha}{\mu+\alpha}-1}{\sqrt\frac{L+\alpha}{\mu+\alpha}+1}}^k \implies
	k\ge\frac{1}{2}\circpar{\sqrt{\frac{L+\alpha}{\mu+\alpha}-1}} (\ln(c) + \ln(1/\epsilon ) )
\end{align*}
\end{lemma}
\begin{proof}
Note that the function
\begin{align*}
	\delta(x) = \ln\frac{\sqrt x -1}{\sqrt x +1} + \frac{2}{\sqrt{x-1} }
\end{align*}
takes non-negative values for $x>1$, as 
\begin{align*}
	\delta'(x) &= \frac{\sqrt x+1}{\sqrt x-1} \frac{ 0.5 x^{-\nicefrac{1}{2}} (\sqrt x+1) - 0.5 x^{-\nicefrac{1}{2}}(\sqrt x-1) }{(\sqrt x+1)^2} - \frac{1}{(x-1)\sqrt{x-1}}\\
	&= \frac{ 1}{(x-1)\sqrt x} - \frac{1}{(x-1)\sqrt{x-1}}<0\\
\end{align*}
and $\lim_{x\to\infty}\delta(x) =0$. Thus, we obtained the following inequality 
\begin{align*}
	\frac{\sqrt{x}-1}{\sqrt x+1}
	&\ge\exp\circpar{\frac{-2}{\sqrt{x-1}}},~x>1,
\end{align*}
yields
\begin{align*}
c\circpar{\frac{\sqrt\frac{L+\alpha}{\mu+\alpha}-1}{\sqrt\frac{L+\alpha}{\mu+\alpha}+1}}^k\ge c\exp\circpar{\frac{-2k}{\sqrt{\frac{L+\alpha}{\mu+\alpha}-1}}}.
\end{align*}
Hence, 
\begin{align*}
	&\ln\epsilon \ge\ln(c) + \frac{-2k}{\sqrt{\frac{L+\alpha}{\mu+\alpha}-1}}\\
	\implies &\frac{2k}{\sqrt{\frac{L+\alpha}{\mu+\alpha}-1}}\ge \ln(c) + \ln(1/\epsilon )\\
	\implies &k\ge \frac{1}{2}\circpar{\sqrt{\frac{L+\alpha}{\mu+\alpha}-1}} (\ln(c) + \ln(1/\epsilon ) )
\end{align*}

\end{proof}

	\end{document}